\documentclass[british,11pt]{extarticle}

\usepackage[T1]{fontenc}
\usepackage[latin9]{inputenc}
\usepackage[a4paper]{geometry}
\usepackage{color}
\usepackage{babel}
\usepackage{float}
\usepackage{dsfont}
\usepackage{amsmath}
\usepackage{amsthm}
\usepackage{amssymb}
\usepackage{stmaryrd}
\usepackage{graphicx}
\usepackage[unicode=true,pdfusetitle,
 bookmarks=true,bookmarksnumbered=false,bookmarksopen=false,
 breaklinks=false,pdfborder={0 0 1},backref=false,colorlinks=true]
 {hyperref}

\makeatletter

\newcommand{\lyxdot}{.}

\newcommand{\lyxaddress}[1]{
	\par {\raggedright #1
	\vspace{1.4em}
	\noindent\par}
}
\theoremstyle{plain}
\newtheorem{thm}{\protect\theoremname}
\theoremstyle{remark}
\newtheorem{rem}[thm]{\protect\remarkname}

\usepackage{tikz-cd}
\usepackage{leftindex}

\makeatother

\providecommand{\remarkname}{Remark}
\providecommand{\theoremname}{Theorem}

\begin{document}
\title{Machine-learning invariant foliations in forced systems for reduced
order modelling}
\author{Robert Szalai}
\date{20th March 2024}
\maketitle

\lyxaddress{School of Engineering Mathematics and Technology, University of Bristol,
Ada Lovelace Building, Tankard's Close, Bristol BS8 1TW, email: r.szalai@bristol.ac.uk}
\begin{abstract}
We identify reduced order models (ROM) of forced systems from data
using invariant foliations. The forcing can be external, parametric,
periodic or quasi-periodic. The process has four steps: 1. identify
an approximate invariant torus and the linear dynamics about the torus;
2. identify a globally defined invariant foliation about the torus;
3. identify a local foliation about an invariant manifold that complements
the global foliation 4. extract the invariant manifold as the leaf
going through the torus and interpret the result. We combine steps
2 and 3, so that we can track the location of the invariant torus
and scale the invariance equations appropriately. We highlight some
fundamental limitations of invariant manifolds and foliations when
fitting them to data, that require further mathematics to resolve.
\end{abstract}

\section{Introduction}

In this paper we develop a numerical method that identifies reduced
order models (ROM) of deterministic and externally forced systems
from data. A ROM is a self-contained system that tracks a small number
of properties of the system over time accurately. Ideally, a ROM represents
dynamics independent of the coordinate system in which the data is
collected. This means that when a ROM is identified, a coordinate
system, where the ROM is minimal, must also be identified \cite{Champion2019Autoencoder}.
There are many uses of a ROM, such as faster model predictions, identification
of governing equations of physical phenomena and meaningful interpretation
of data \cite{CAMPSVALLS20231,epstein2008}.

In our definition, a genuine ROM discards unimportant dynamics \cite{Szalai2023Fol},
hence it does not reproduce the full data set. Instead, the ROM picks
out a specific and simple phenomenon. A data set can yield many different
ROMs, depending on what phenomenon we wish to extract. Only when all
phenomena are identified can one reproduce the full data set. The
coupling and relations among all identified ROMs can be recovered
using the identified coordinate systems.

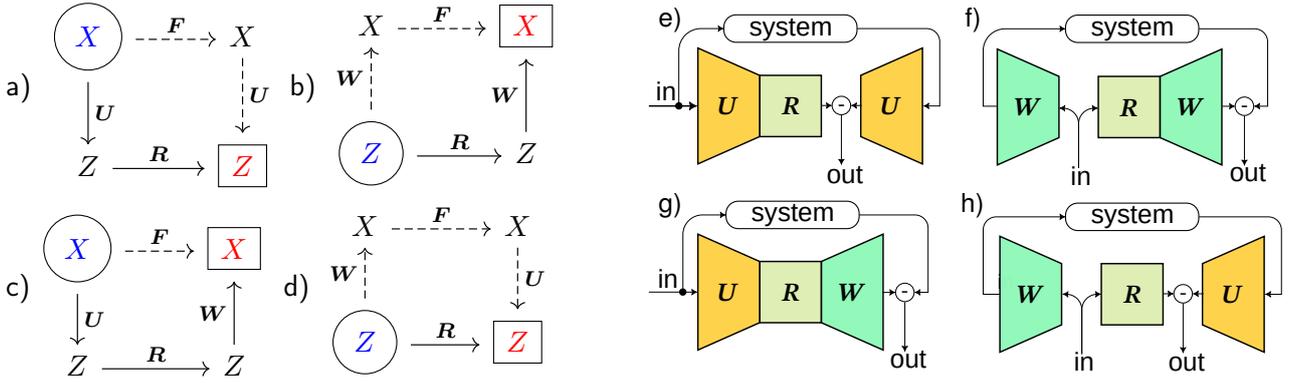
\begin{figure}[H]
\begin{centering}
\begin{center}
\begin{minipage}{0.49\textwidth}
a)
\begin{tikzcd} \tikz \node[draw,circle]{\textcolor{blue}{$X$}}; \arrow[r, dashed, "\boldsymbol{F}"] \arrow[d, "\boldsymbol{U}"] & X \arrow[d, dashed, "\boldsymbol{U}"]  \\ 
Z \arrow[r, "\boldsymbol{R}"]& \tikz \node[draw,rectangle]{\textcolor{red}{$Z$}}; \end{tikzcd}
b) 
\begin{tikzcd} X \arrow[r, dashed, "\boldsymbol{F}"] & \tikz \node[draw,rectangle]{\textcolor{red}{$X$}};   \\ 
\tikz \node[draw,circle]{\textcolor{blue}{$Z$}}; \arrow[r, "\boldsymbol{R}"]  \arrow[u,  dashed, "\boldsymbol{W}"]& Z \arrow[u, "\boldsymbol{W}"] \end{tikzcd}\\
c)\begin{tikzcd} \tikz \node[draw,circle]{\textcolor{blue}{$X$}}; \arrow[r, dashed, "\boldsymbol{F}"] \arrow[d, "\boldsymbol{U}"] & \tikz \node[draw,rectangle]{\textcolor{red}{$X$}}; \\
Z \arrow[r, "\boldsymbol{R}"]                             & Z \arrow[u, "\boldsymbol{W}"]                             \end{tikzcd} d)
\begin{tikzcd} X \arrow[r, dashed,  "\boldsymbol{F}"]                              & X  \arrow[d, dashed, "\boldsymbol{U}"]                             \\ 
\tikz \node[draw,circle]{\textcolor{blue}{$Z$}}; \arrow[r, "\boldsymbol{R}"] \arrow[u, dashed, "\boldsymbol{W}"] & \tikz \node[draw,rectangle]{\textcolor{red}{$Z$}}; \end{tikzcd}%
\end{minipage}
\begin{minipage}{0.49\textwidth}
\includegraphics[width=\textwidth]{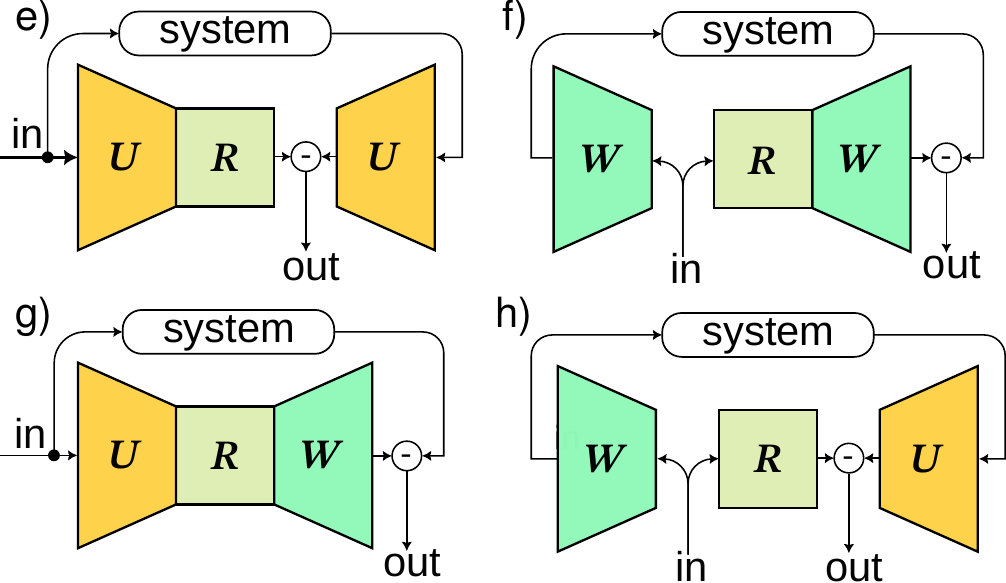}
\end{minipage}
\par\end{center}
\par\end{centering}
\caption{\label{fig:straight-commutative}Commutative diagrams of a) invariant
foliations, b) invariant manifolds, and connection diagrams of c)
autoencoders and d) equation-free models. The dashed arrows denote
the chain of function composition(s) that involves $\boldsymbol{F}$,
the continuous arrows denote the chain of function composition(s)
that involves map $\boldsymbol{R}$. The encircled vector space denotes
the domain of the defining equation and the boxed vector space is
the target of the defining equation. Diagrams (e,f,g,h) represent
the defining equations corresponding to diagrams (a,b,c,d). The input
(in) is the domain of the defining equation and the output (out) is
the target, which must vanish.}
\end{figure}

The mathematical term for having a self-contained part of the system
is \emph{invariance}. Invariance means that the evolution of the system
and the ROM are the same after a possibly non-invertible, but either
surjective or injective transformation \cite[Definition 1]{Szalai2023Fol}.
The transformation can either map from the physical coordinates to
the state space of the ROM, which is called the encoder or in reverse,
which is called the decoder. If we think of a system that maps an
initial condition to a prediction, we need to relate the state of
the system and the ROM to both the initial condition and the prediction.
Given that we have a choice between an encoder and decoder for relating
the ROM and the system, this gives us four architectures for ROM identification.
The four possibilities are illustrated in figure \ref{fig:straight-commutative}
and are called invariant foliations \cite{Szalai2020ISF}, invariant
manifolds \cite{hirsch1970}, autoencoders \cite{Kramer1991autoencoder}
and equation-free models \cite{Kevrekidis2003}. In what follows,
we represent the encoder by function $\boldsymbol{U}$, the decoder
by $\boldsymbol{W}$, the ROM by $\boldsymbol{R}$ and the (unknown)
system by $\boldsymbol{F}$. Often we cannot manipulate the initial
conditions of our system, which precludes us from using a decoder
when comparing initial conditions. This leaves us with invariant foliations
and autoencoders to fit the data to. However, as discussed in \cite{Szalai2023Fol},
autoencoders cannot ensure invariance except when the dimensionality
of the system and the dimensionality of the ROM are the same. Hence,
only invariant foliations remain a possibility, which we use subsequently.

Invariant foliations are not a new idea \cite{BatesFoliations2000,ShubFoliation2012,Lawson1974},
they are used to analyse chaotic systems \cite{wiggins2003introduction}.
Most similar to our application is, where an invariant foliation was
used to find initial conditions for ROMs \cite{Roberts89}. Elsewhere,
invariant foliations are also called fibre bundles \cite{AulbachFiber1998}.
For data-driven reduced order modelling about fixed points, non-resonant
invariant foliations were introduced in \cite{Szalai2020ISF,Szalai2023Fol},
and their existence, uniqueness was proven using the parametrisation
method \cite{delaLlave1997,Haro2016}. The geometric interpretation
of an invariant foliation is explained in \cite{Lawson1974,Szalai2020ISF}.

In many circumstances it is necessary to identify an invariant manifold,
represented by a decoder, that maps trajectories of the ROM back to
the physical space, where that data came from. An invariant manifold
can be calculated as a level set of an encoder of an invariant foliation
that contains the equilibrium or (quasi-) periodic orbit. Therefore,
we identify two invariant foliations from our data: one that captures
the dynamics of interest and one that complements it. The accuracy
of the second foliation is only important near the invariant manifold,
which allows us to use a low-order (possibly linear) model that maps
the leaves into each other. Having two complementary foliations also
allows us to recover the equilibrium or (quasi-) periodic orbit, which
lies at the intersection of the zero level sets of the two encoders.

The invariant foliation that captures the slowest dynamics is within
the class of once differentiable foliations, which allows for an accurate
ROM identification. However, the complementary foliation -- and consequently
the invariant manifold it defines -- is only unique among foliations
that are many times differentiable, depending on the spectral properties
of the system. Moreover, the numerical error of calculating the invariant
manifold scales with the power of the required order of differentiability
for uniqueness. To make this point precise let us consider the linear
system 
\begin{align*}
\dot{x} & =-x\\
\dot{y} & =-\frac{11}{2}y
\end{align*}
with trajectories $y$$\left(x\right)=cx^{11/2}$, where $c$ is an
integration constant. All trajectories are invariant manifolds, but
the only smooth one is $y\left(x\right)=0$. If we make a small numerical
error of $10^{-11}$ near the equilibrium, such that $y\left(10^{-2}\right)=10^{-11}$,
the error will be on the trajectory with $c=1$, which gives an error
at unit distance from the equilibrium of $y\left(1\right)=1$. The
error is not due to lack of invariance, because the result is an exact
solution, instead we have chosen a manifold that is only five times
differentiable as opposed to the required six. We used the fractional
coefficient $\frac{11}{2}$ to avoid talking about resonances.

As opposed to methods that carry out an asymptotic expansion of an
invariant manifold about an invariant torus or equilibrium, we fit
the foliation to data points and therefore the invariance equation
is satisfied approximately at each data point. This can result in
great differences from an asymptotically calculated invariant manifold,
the resulting ROM is still accurate as it satisfies the invariance
equation, sometimes better than the asymptotic expansion.

We note that there are other ways to single out a unique invariant
manifold than smoothness, for example by exponential dichotomy \cite{Llave1995}.
The two definitions do not always coincide \cite{delaLlave1997}.
However, both approaches are local to the equilibrium or invariant
torus, as opposed to data fitting, which is global. Data fitting cannot
take smoothness or exponential dichotomy into account, because data
points are discrete and we cannot differentiate over them. It is also
unclear what a global minimum of a loss function represents and how
it relates to asymptotic uniqueness criteria. It remains to be seen
whether a global criterion that defines unique and meaningful invariant
manifolds or foliations is possible.

The present paper is limited to a single set of parameters and therefore
not suitable for bifurcation analysis. A trivial, parameter dependent
extension, designates the parameters as state variables without any
dynamics and adds some constraints to the numerical representation
of the foliation.

The paper has three major sections. Section \ref{sec:theory} summarises
the relevant theory, section \ref{sec:numerics} introduces the numerical
methods and section \ref{sec:examples} illustrates the method on
a number of example. The impatient reader can skip to section \ref{sec:numerics}.

\section{\label{sec:theory}Set-up and invariant foliations}

We assume a deterministic system and that the state of the system
is sampled uniformly in time, giving us a series of data points. One
part of the state is a real, $n$-dimensional inner product vector
space $X$. The other part of the state is the $d$-dimensional torus
$\mathbb{T}^{d}$, where the forcing occurs due to a rigid rotation
$\boldsymbol{\omega}\in\mathbb{T}^{d}$. The evolution of the state
is described by a real analytic function $\boldsymbol{F}:\mathbb{T}^{d}\times X$,
such that
\begin{equation}
\begin{array}{rl}
\boldsymbol{x}_{k+1} & \negthickspace\negthickspace\negthickspace=\boldsymbol{F}\left(\boldsymbol{x}_{k},\boldsymbol{\theta}_{k}\right)\\
\boldsymbol{\theta}_{k+1} & \negthickspace\negthickspace\negthickspace=\boldsymbol{\theta}_{k}+\boldsymbol{\omega}
\end{array},\;k=1,2,\ldots.\label{eq:MAPSysDef-1}
\end{equation}

An invariant foliation is always specific to an underlying invariant
object, which in our case is an invariant torus. The invariant torus
can be parameterised over the torus $\mathbb{T}^{d},$ which we denote
by
\begin{equation}
\mathcal{T}=\left\{ \boldsymbol{K}\left(\boldsymbol{\theta}\right):\boldsymbol{\theta}\in\mathbb{T}^{d}\right\} ,\label{eq:TOR-geometry}
\end{equation}
where $\boldsymbol{K}:\mathbb{T}^{d}\to X$ is an analytic function.
The torus $\mathcal{T}$ is invariant if the invariance equation
\begin{equation}
\boldsymbol{K}\left(\boldsymbol{\theta}+\omega\right)=\boldsymbol{F}\left(\boldsymbol{\theta},\boldsymbol{K}\left(\boldsymbol{\theta}\right)\right)\label{eq:TOR-invariances}
\end{equation}
holds. Our numerical implementation uses $d=1$, in which case the
torus is a closed curve.

\subsection{Linear dynamics}

The calculation of invariant foliations depends on the linear dynamics
about the invariant torus. Let us define
\[
\boldsymbol{A}\left(\boldsymbol{\theta}\right)=D_{1}\boldsymbol{F}\left(\boldsymbol{K}\left(\boldsymbol{\theta}\right),\boldsymbol{\theta}\right)
\]
and consider the linearised system 
\begin{equation}
\begin{array}{rl}
\boldsymbol{x}_{k+1} & \negthickspace\negthickspace\negthickspace=\boldsymbol{A}\left(\boldsymbol{\theta}_{k}\right)\boldsymbol{x}_{k}\\
\boldsymbol{\theta}_{k+1} & \negthickspace\negthickspace\negthickspace=\boldsymbol{\theta}_{k}+\boldsymbol{\omega}
\end{array},\;k=1,2,\ldots.\label{eq:ED-linsys-1}
\end{equation}
Instead of eigenvectors and eigenvalues, we need to use invariant
vector bundles to decompose the dynamics of the linear system \ref{eq:ED-linsys-1}.
Invariant foliations require the use of left vector bundles that satisfy
the linear invariance equation 
\begin{equation}
\boldsymbol{\Lambda}_{j}\left(\boldsymbol{\theta}\right)\boldsymbol{U}_{j}\left(\boldsymbol{\theta}\right)=\boldsymbol{U}_{j}\left(\boldsymbol{\theta}+\boldsymbol{\omega}\right)\boldsymbol{A}\left(\boldsymbol{\theta}\right),\label{eq:ED-left-bundle}
\end{equation}
where $\boldsymbol{U}_{j}:\mathbb{T}^{d}\to\mathcal{L}\left(X,Z_{j}\right)$
and $\boldsymbol{\Lambda}_{j}:\mathbb{T}^{d}\to\mathcal{L}\left(Z_{j},Z_{j}\right)$
are analytic matrix valued functions and $Z_{j}$ are a low-dimensional
vector spaces. A complete decomposition means that $X$ is isomorphic
to $\bigotimes_{j=1}^{m}Z_{j}$, and $\bigoplus_{j=1}^{m}\left(\ker\boldsymbol{U}_{j}\left(\boldsymbol{\theta}\right)\right)^{\perp}=X$
for all $\boldsymbol{\theta}\in\mathbb{T}^{d}$. To characterise the
linear dynamics, we use exponential dichotomies. The matrix $\boldsymbol{\Lambda}_{j}$
has an exponential dichotomy for $\rho\in\mathbb{R}^{+}$ if there
exists $C>0$ such that the inequalities 
\begin{align*}
\left|\boldsymbol{\Lambda}_{j}\left(\boldsymbol{\theta}+\left(k-1\right)\boldsymbol{\omega}\right)\cdots\boldsymbol{\Lambda}_{j}\left(\boldsymbol{\theta}+\boldsymbol{\omega}\right)\boldsymbol{\Lambda}_{j}\left(\boldsymbol{\theta}\right)\right| & \le C\rho^{k}, & \text{and}\\
\left|\boldsymbol{\Lambda}_{j}^{-1}\left(\boldsymbol{\theta}-k\boldsymbol{\omega}\right)\cdots\boldsymbol{\Lambda}_{j}^{-1}\left(\boldsymbol{\theta}-\boldsymbol{\omega}\right)\right| & \le C\rho^{-k}, & \forall k\in\mathbb{N}
\end{align*}
hold. Our spectral decomposition is such that $\boldsymbol{\Lambda}_{j}$
do not have exponential dichotomy when $\rho\in\Sigma_{j}=\left[\alpha_{j},\beta_{j}\right]$
but does anywhere else in $\mathbb{R}^{+}$. The sets $\Sigma_{j}=\left[\alpha_{j},\beta_{j}\right]$
are called spectral intervals, that are pairwise disjoint, i.e. $\Sigma_{j}\cap\Sigma_{k}=\emptyset$
for $j\neq k$. We denote the number of spectral intervals by $m$.

For a constant matrix $\boldsymbol{A}$, the spectral intervals reduce
to points ($\alpha_{j}=\beta_{j}$), which are the magnitudes of the
eigenvalues of $\boldsymbol{A}$. In addition, the row vectors of
$\boldsymbol{U}_{j}$ span the same space as the left eigenvectors
of $\boldsymbol{A}$ corresponding to the eigenvalues that have the
magnitude $\alpha_{j}=\beta_{j}$. In section \ref{subsec:Bundles}
we detail how to find $\boldsymbol{\Lambda}_{j}$ and $\boldsymbol{U}_{j}$
numerically by discretising \eqref{eq:ED-left-bundle} and turning
it into an eigenvalue problem.

\subsection{\label{subsec:Foliations}Invariant foliation}

An invariant foliation is the solution of the invariance equation
\begin{equation}
\boldsymbol{R}\left(\boldsymbol{U}\left(\boldsymbol{x},\boldsymbol{\theta}\right),\boldsymbol{\theta}\right)=\boldsymbol{U}\left(\boldsymbol{F}\left(\boldsymbol{x},\boldsymbol{\theta}\right),\boldsymbol{\theta}+\boldsymbol{\omega}\right),\label{eq:FOIL-invariance}
\end{equation}
which also corresponds to the commutative diagram in figure \ref{fig:straight-commutative}(a).
The functions $\boldsymbol{R}:Z\times\mathbb{T}^{d}\to Z$ and $\boldsymbol{U}:X\times\mathbb{T}^{d}\to Z$
in \eqref{eq:FOIL-invariance} are real analytic. In order to identify
our foliation of interest, we select a number of spectral intervals
using the index set 
\begin{equation}
\mathcal{\mathcal{I}}=\left\{ i_{1},i_{2},\ldots,i_{\#\mathcal{I}}\right\} \label{eq:ED-index-set}
\end{equation}
and stipulate that
\begin{equation}
D\boldsymbol{R}\left(\boldsymbol{0},\boldsymbol{\theta}\right)=\begin{pmatrix}\boldsymbol{\Lambda}_{i_{1}}\left(\boldsymbol{\theta}\right) &  & \boldsymbol{0}\\
 & \ddots\\
\boldsymbol{0} &  & \boldsymbol{\Lambda}_{i_{\#\mathcal{I}}}\left(\boldsymbol{\theta}\right)
\end{pmatrix}\;\text{and}\;D\boldsymbol{U}\left(\boldsymbol{K}\left(\boldsymbol{\theta}\right),\boldsymbol{\theta}\right)=\begin{pmatrix}\boldsymbol{U}_{i_{1}}\left(\boldsymbol{\theta}\right)\\
\vdots\\
\boldsymbol{U}_{i_{\#\mathcal{I}}}\left(\boldsymbol{\theta}\right)
\end{pmatrix}.\label{eq:FOIL-BC}
\end{equation}
This means that the vector space of the conjugate dynamics $\boldsymbol{R}$
is $Z=\bigotimes_{j\in\mathcal{I}}Z_{j}$. We also make the normalising
assumption that $\boldsymbol{U}\left(\boldsymbol{K}\left(\boldsymbol{\theta}\right),\boldsymbol{\theta}\right)=\boldsymbol{0}$,
from which it follows that $\boldsymbol{R}\left(\boldsymbol{0},\boldsymbol{\theta}\right)=\boldsymbol{0}$.
In what follows, we concurrently calculate the complementary foliation
for index set $\mathcal{I}^{c}=\left\{ 1,\ldots,m\right\} \setminus\mathcal{I}$
using the invariance equation 
\begin{equation}
\boldsymbol{S}\left(\boldsymbol{V}\left(\boldsymbol{x},\boldsymbol{\theta}\right),\boldsymbol{\theta}\right)=\boldsymbol{V}\left(\boldsymbol{F}\left(\boldsymbol{x},\boldsymbol{\theta}\right),\boldsymbol{\theta}+\boldsymbol{\omega}\right)\label{eq:FOIL-compl-inv}
\end{equation}
with similar boundary conditions except that $\mathcal{I}$ is replaced
by $\mathcal{I}^{c}$. We also define the vector space $Z^{c}=\bigotimes_{j\in\mathcal{I}^{c}}Z_{j}$
and therefore the functions in equation \eqref{eq:FOIL-compl-inv}
are $\boldsymbol{S}:Z^{c}\times\mathbb{T}^{d}\to Z^{c}$ and $\boldsymbol{V}:X\times\mathbb{T}^{d}\to Z^{c}$.

The functions $\boldsymbol{U}$ and $\boldsymbol{R}$ are not unique,
even when the foliation itself is unique. This is because any invertible
transformation $\boldsymbol{T}:Z\times\mathbb{T}^{d}\to Z$ can create
a new encoder $\tilde{\boldsymbol{U}}\left(\boldsymbol{x},\boldsymbol{\theta}\right)=\boldsymbol{T}\left(\boldsymbol{U}\left(\boldsymbol{x},\boldsymbol{\theta}\right),\boldsymbol{\theta}\right)$
and conjugate map $\tilde{\boldsymbol{R}}\left(\boldsymbol{z},\boldsymbol{\theta}\right)=$$\boldsymbol{T}\left(\boldsymbol{R}\left(\boldsymbol{T}^{-1}\left(\boldsymbol{x},\boldsymbol{\theta}\right),\boldsymbol{\theta}\right),\boldsymbol{\theta}+\boldsymbol{\omega}\right)$,
which satisfy the invariance equation \eqref{eq:FOIL-invariance}.
A simple restriction to make the solution of \eqref{eq:FOIL-invariance}
unique requires that 
\begin{equation}
\boldsymbol{U}\left(\boldsymbol{W}\left(\boldsymbol{\mathbf{\theta}}\right)\boldsymbol{z}\right)=\boldsymbol{z},\label{eq:FOIL-graph-constr}
\end{equation}
where $\boldsymbol{W}$ satisfies the invariance equation 
\begin{equation}
\boldsymbol{A}\left(\boldsymbol{\theta}\right)\boldsymbol{W}\left(\boldsymbol{\theta}\right)=\boldsymbol{W}\left(\boldsymbol{\theta}+\boldsymbol{\omega}\right)D\boldsymbol{R}\left(\boldsymbol{0},\boldsymbol{\theta}\right)\label{eq:ED-righ-bundle}
\end{equation}
or just simply $\boldsymbol{W}\left(\boldsymbol{\mathbf{\theta}}\right)D\boldsymbol{V}\left(\boldsymbol{K}\left(\boldsymbol{\theta}\right),\boldsymbol{\theta}\right)=\boldsymbol{0}$.
It is not necessary that $\boldsymbol{W}$ is accurately calculated,
as the only requirement is that $\boldsymbol{W}\left(\boldsymbol{\mathbf{\theta}}\right)D\boldsymbol{U}\left(\boldsymbol{K}\left(\boldsymbol{\theta}\right),\boldsymbol{\theta}\right)$
is invertible for all $\boldsymbol{\theta}\in\mathbb{T}^{d}$. In
what follows, we will only restrict the nonlinear part of $\boldsymbol{U}$,
so that $\boldsymbol{U}_{nl}\left(\boldsymbol{W}\left(\boldsymbol{\mathbf{\theta}}\right)\boldsymbol{z}\right)=\boldsymbol{0}$
for an approximate $\boldsymbol{W}$. Using $\boldsymbol{W}$ that
satisfies \eqref{eq:ED-righ-bundle} leads to a so-called graph-style
parametrisation akin to how invariant manifolds are parametrised in
\cite[Theorem 1.2]{CabreLlave2003}.
\begin{rem}
It is important to think about the required value for parameter $d$,
because an unnecessarily large number increases computational costs.
A periodically forced system needs a single angle variable if the
forcing frequency is not an integer multiple of the sampling frequency.
However if the forcing frequency is an integer multiple of the sampling
frequency, there is no need for an angle variable at all. Similarly,
when there are multiple rationally unrelated forcing frequencies and
one of the forcing frequencies is an integer multiple of the sampling
frequency, we can use one less angle variables than the number of
forcing frequencies. In reality, it might not be possible or desirable
to synchronise forcing and sampling. In order to reproduce finer details
of the signal and resolve higher harmonics, the sampling frequency
cannot be a fraction of the forcing frequency, it must be at least
twice as high as the lowest resolved frequency component, according
to the Shannon-Nyquist theorem \cite{ShannonNyquist}. In these cases
$d$ must equal the number of rationally unrelated forcing frequencies.
The case of a single forcing frequency is however special, because
regardless of sampling, Floquet theory \cite{chicone2008ordinary}
applies to the underlying continuous-time system. This means that
a time-dependent coordinate transformation can eliminate the angle
variable from the linear part of \eqref{eq:MAPSysDef-1} about the
a periodic orbit. Making the nonlinear terms autonomous requires non-resonance
conditions that we outline later.
\end{rem}

A foliation might not exist if there are resonances among the spectral
intervals. Even though we have assumed an analytic system, an invariant
foliation can be finitely many times differentiable or only continuous.
The theorem below states that among all possible foliations there
is always an analytic one, which is unique.
\begin{thm}
Assume an invariant torus \eqref{eq:TOR-geometry} and a linearised
system about the torus in the form of \eqref{eq:ED-linsys-1}. Also
assume that the linear system has dichotomy spectral intervals $\Sigma_{j}=\left[\alpha_{j},\beta_{j}\right]$,
$1\le j\le m$ and that $\max\beta_{j}<1$. Pick one or more spectral
intervals using a non-empty index set 
\[
\mathcal{I}\subset\left\{ 1,2,\ldots,m\right\} 
\]
with the condition that $\alpha_{j}\neq0$ for all $j\in\mathcal{I}$
so that linear system \eqref{eq:ED-linsys-1} restricted to the subset
of the spectrum 
\[
\Sigma_{\mathcal{F}}=\bigcup_{j\in\mathcal{I}}\Sigma_{j}
\]
is invertible. Define the spectral quotient as 
\[
\beth_{\mathcal{I}}=\frac{\min_{j\in\mathcal{I}}\log\alpha_{j}}{\log\beta_{m}}.
\]
If the non-resonance conditions 
\begin{equation}
1\notin\left[\beta_{i_{0}}^{-1}\alpha_{i_{1}}\cdots\alpha_{i_{j}},\alpha_{i_{0}}^{-1}\beta_{i_{1}}\cdots\beta_{i_{j}}\right]\label{eq:FOIL-non-resonance-1}
\end{equation}
hold for $i_{0}\in\mathcal{I}$, $i_{1},\ldots,i_{j}\in\left\{ 1,2,\ldots,m\right\} $
and $2\le j<\beth_{\mathcal{I}}+1$ with some $i_{k}\notin\mathcal{I}$
then
\begin{enumerate}
\item there exists a unique invariant foliation defined by analytic functions
$\boldsymbol{U}$ and $\boldsymbol{R}$ satisfying equation \eqref{eq:FOIL-invariance}
in a sufficiently small neighbourhood of the torus $\mathcal{T}$
such that equations \eqref{eq:FOIL-BC} hold;
\item the nonlinear map $\boldsymbol{R}$ is a polynomial, which in its
simplest form contains terms for which the internal non-resonance
conditions \eqref{eq:FOIL-non-resonance-1} with $i_{0},i_{1},\ldots,i_{j}\in\mathcal{I}$
and $2\le j<\beth_{\mathcal{I}}+1$ does not hold.
\end{enumerate}
\end{thm}

\begin{proof}
The details can be found in the paper \cite{SzalaiForcedTheory2024}.
\end{proof}

\section{\label{sec:numerics}The data and identifying the foliation}

This section describes the numerical methods that lead to the identification
and analysis of our ROM. To simplify the numerical method, we assume
a one-dimensional torus. It is possible to use more frequencies at
the expense of higher computational costs. The data we consider come
in triplets $\boldsymbol{x}^{k}$, $\boldsymbol{y}^{k}$, $\theta^{k}$,
such that 
\begin{equation}
\boldsymbol{y}^{k}=\boldsymbol{F}\left(\boldsymbol{x}^{k},\theta^{k}\right)+\boldsymbol{\xi}^{k},\qquad k=1\ldots N,\label{eq:DATA-create}
\end{equation}
where $\boldsymbol{\xi}^{k}$ is a small perturbation drawn from a
probability distribution with zero mean. Equation \eqref{eq:DATA-create}
is the unknown part of the dynamics, i.e., map $\boldsymbol{F}$ is
unknown. We assume that the external forcing is known, which is the
second line of equation \eqref{eq:MAPSysDef-1} with a fixed $\omega\in\mathbb{T}$.
When considering trajectories for some set of consecutive indices
we have $\boldsymbol{x}^{k+1}=\boldsymbol{y}^{k}$ and $\theta^{k+1}=\theta^{k}+\omega$.

The notation we use from now on for matrix multiplications and tensor
contractions is the following. A matrix multiplication $\boldsymbol{C}=\boldsymbol{A}\boldsymbol{B}$
is denoted by $C_{il}=A_{ij}B_{jl}$. So whenever an index appears
more than once on one side of an equation it is summed over. When
we do not want to sum over an index, even if it appears in multiple
factors, we underline the specific index, for example $C_{ilk}=A_{ij\underline{k}}B_{jl\underline{k}}$.
The index notation implicitly assumes the existence of an orthonormal
basis, for example a vector can be written as $\boldsymbol{x}=x_{i}\boldsymbol{e}_{i}\in X$,
where $\boldsymbol{e}_{i}\in X$, $i\in\left\{ 1,\ldots,n\right\} $
is a set of orthonormal basis vectors. Orthonormality allows us to
use the same vectors for the dual basis, which greatly simplifies
our notation.

\subsection{Discretisation and shift along the torus}

We use Fourier collocation \cite{trefethen} to resolve functions
on the torus $\mathbb{T}$. Fourier collocation demands a uniform
grid, which is given by the nodes

\begin{equation}
\vartheta_{1}=0,\ldots,\vartheta_{2\ell+1}=\frac{2\ell}{2\ell+1}2\pi,\label{eq:GRID}
\end{equation}
where $\ell$ is the number of harmonics to be resolved. A frequency
limited function $\boldsymbol{x}:\mathbb{T}\to\mathbb{R}^{n}$ can
be reconstructed using 
\[
\boldsymbol{x}\left(\theta\right)=\sum_{j=1}^{2\ell+1}\gamma\left(\theta-\vartheta_{j}\right)\boldsymbol{x}_{j},
\]
where $\boldsymbol{x}_{j}=\boldsymbol{x}\left(\vartheta_{j}\right)$
and
\[
\gamma\left(\theta\right)=\frac{1}{2\ell+1}\frac{\sin\frac{2\ell+1}{2}\theta}{\sin\frac{1}{2}\theta}.
\]
A function can also be represented by a two-index array $x_{ij}=x_{i}\left(\vartheta_{j}\right)$,
where $i$ is the coordinate index of $\boldsymbol{x}$ in our basis
$\boldsymbol{e}_{i}$. Interpolating a function at data points $\theta^{k}$
can be written as $x_{i}\left(\theta^{k}\right)=x_{ij}t_{j}^{k},$where
$t_{j}^{k}=\gamma\left(\theta^{k}-\vartheta_{j}\right)$. Similarly,
interpolating a matrix-valued function $\boldsymbol{A}:\mathbb{T}\to\mathbb{R}^{n\times n}$
at the grid points can be written as $A_{ij}\left(\theta^{k}\right)=A_{ijl}t_{l}^{k}$.

A common operation with forced discrete-time systems is to shift a
signal to the right by the angle $\omega$. The shift operator $\mathcal{S}^{\omega}$
is defined as
\[
\left(\mathcal{S}^{\omega}\boldsymbol{x}\right)\left(\theta\right)=\boldsymbol{x}\left(\theta-\omega\right).
\]
The discrete version of the shift operator is calculated from 
\begin{align*}
\boldsymbol{y}_{i}=\left(\mathcal{S}^{\omega}\boldsymbol{x}\right)\left(\vartheta_{i}\right) & =\sum_{j=1}^{2\ell+1}\gamma\left(\vartheta_{i}-\vartheta_{j}-\omega\right)\boldsymbol{x}_{j}.
\end{align*}
We then introduce the notation 
\[
\mathds{S}_{jk}^{\omega}=\gamma\left(\vartheta_{k}-\vartheta_{j}-\omega\right),
\]
so that a shifted representation of a function becomes $y_{il}=x_{ij}\mathds{S}_{jl}^{\omega}$.
Note that $\left(\mathds{S}^{\omega}\right)^{T}=\mathds{S}^{-\omega}$,
because function $\gamma$ is even.

\subsection{\label{subsec:LINID}Approximate invariant torus and nearby linear
dynamics}

In order to find the torus and the linear dynamics about the torus,
we assume the following model

\[
\boldsymbol{F}\left(\boldsymbol{x},\theta\right)=\boldsymbol{A}\left(\theta\right)\boldsymbol{x}+\boldsymbol{b}\left(\theta\right),
\]
where 
\[
\boldsymbol{A}\left(\boldsymbol{\theta}\right)=\sum_{k=1}^{2\ell+1}\gamma\left(\theta-\vartheta_{k}\right)\boldsymbol{A}_{k}.
\]
The linear model identification is iterative: we first use all data
points and then remove data that is far from the initial guess of
the torus to refine both the torus and the linear model near the torus.
We also use scaling so that data points closer to the torus bear greater
importance than data further from the torus. The initial model parameters
are found by solving the optimisation problem
\begin{equation}
\arg\min_{\boldsymbol{A},\boldsymbol{b}}\frac{1}{2}\sum_{k=1}^{N}\left\Vert \boldsymbol{A}\left(\theta^{k}\right)\boldsymbol{x}^{k}+\boldsymbol{b}\left(\theta^{k}\right)-\boldsymbol{y}^{k}\right\Vert ^{2}.\label{eq:LINID-noscale}
\end{equation}
After a solution to \eqref{eq:LINID-noscale} we remove some of the
data that are not close to the torus in Euclidean norm (the particular
amount is problem dependent). The value of the norm in \eqref{eq:LINID-noscale}
vanishes on the invariant torus and therefore the fitting error near
the torus would be assigned less importance without scaling. In order
to make the significance of the relative fitting error approximately
uniform over the data we introduce a scaling factor for each data
point
\begin{equation}
\delta^{k}=\frac{1}{\epsilon^{2}+\left\Vert \boldsymbol{x}^{k}-\boldsymbol{K}\left(\theta^{k}\right)\right\Vert ^{2}},\label{eq:LINID-scalefactor}
\end{equation}
where $\epsilon>0$ is a small number, which we choose to be $\epsilon=2^{-8}$.
Factors $\delta^{k}$ are updated after each iteration, otherwise
they are considered constants. We apply the scaling factor $\delta^{k}$
to each term in \eqref{eq:LINID-noscale} and re-write it using our
index notation into 
\begin{equation}
\arg\min_{\boldsymbol{A},\boldsymbol{b}}\frac{1}{2}\sum_{k=1}^{N}\delta^{k}\left\Vert A_{ijl}x_{j}^{k}t_{l}^{k}+b_{il}t_{l}^{k}-y_{i}^{k}\right\Vert ^{2},\label{eq:LINID-scale}
\end{equation}
which is a linear regression problem and can be solved directly. To
elaborate this point, we define $\tilde{A}_{ijl}=A_{ijl}$, $\tilde{A}_{i\left(n+1\right)l}=b_{il}$
and $\tilde{x}_{jl}^{k}=x_{j}^{k}t_{l}^{k}$, $\tilde{x}_{\left(n+1\right)l}^{k}=t_{l}^{k}$
so that the loss function becomes 
\[
\arg\min_{\boldsymbol{A},\boldsymbol{b}}\frac{1}{2}\sum_{k=1}^{N}\delta^{k}\left|\tilde{A}_{ijl}\tilde{x}_{jl}^{k}-y_{i}^{k}\right|^{2},
\]
which is a standard least squares problem \cite{boyd_vandenberghe_2018}
and has the solution of 
\[
\tilde{\boldsymbol{A}}=\boldsymbol{Y}\boldsymbol{X}^{-1},
\]
where 
\begin{align*}
X_{ijpq} & =N^{-1}\sum_{k=1}^{N}\delta^{k}\tilde{x}_{ij}^{k}\tilde{x}_{pq}^{k},\\
Y_{ipq} & =N^{-1}\sum_{k=1}^{N}\delta^{k}y_{i}^{k}\tilde{x}_{pq}^{k}.
\end{align*}
The iteration stops when the update to $\tilde{\boldsymbol{A}}$ is
smaller than machine precision.

The torus represented by function $\boldsymbol{K}$ is calculated
from the invariance equation
\[
\boldsymbol{K}\left(\theta+\omega\right)=\boldsymbol{A}\left(\boldsymbol{\theta}\right)\boldsymbol{K}\left(\theta\right)+\boldsymbol{b}\left(\boldsymbol{\theta}\right)
\]
that can be written using index notation as
\begin{align}
K_{ij}\mathds{S}_{jl}^{-\omega} & =A_{ip\underline{l}}K_{p\underline{l}}+b_{il}.\label{eq:LINID-torus-tens}
\end{align}
Equation \eqref{eq:LINID-torus-tens} is a Sylvester equation \cite{BartelsStewartAlg432}
and can be further transformed into the matrix-vector equation 
\[
\left(\delta_{ip}\mathds{S}_{lq}^{\omega}-A_{ipq}\right)K_{pq}=b_{il}
\]
that is solved using standard techniques for $\boldsymbol{K}$ that
represents the invariant torus.

\subsection{\label{subsec:Bundles}Invariant vector bundles about the torus}

We now want to solve the invariance equation \eqref{eq:ED-left-bundle}
on the grid \eqref{eq:GRID}, which can be done by calculating the
eigenvalues and eigenvectors of a discretised variant of \eqref{eq:ED-left-bundle}.
Such calculation assumes that $\boldsymbol{\Lambda}_{j}$ is constant.
Note that we can only assume a constant $\boldsymbol{\Lambda}_{j}$,
when system \eqref{eq:ED-linsys-1} is reducible. It is however difficult
to tell if a system is reducible numerically, so we will only observe
the numerics breaking down. When fitting the invariant foliation to
data, we mitigate problems arising from near irreducibility by normalising
the calculated vector invariant bundle, that eventually makes $\boldsymbol{\Lambda}_{j}$
non-autonomous.

The discretised version of \eqref{eq:ED-left-bundle} is written as
\begin{equation}
\lambda\boldsymbol{u}^{T}\left(\vartheta_{j}\right)=\boldsymbol{u}^{T}\left(\vartheta_{j}+\omega\right)\boldsymbol{A}\left(\vartheta_{j}\right),\label{eq:BUND-eigvals}
\end{equation}
which translates to the eigenvector-eigenvalue problem
\begin{align}
\left(\lambda\delta_{il}\delta_{jk}-\mathds{S}_{j\underline{k}}^{-\omega}A_{il\underline{k}}\right)u_{ij} & =0,\label{eq:BUND-tens-eigvals}
\end{align}
where $\lambda$ is an eigenvalue and $u_{ij}$ is an eigenvector.
The eigenvalues are approximately placed along concentric circles
in the complex plane and we expect that each circle contains an integer
multiple of $2\ell+1$ eigenvalues. Using k-nearest neighbours \cite{biau2015lectures},
we attempt to find $n_{cl}=n$ clusters for the magnitudes of the
eigenvalues $\left|\lambda\right|$, expecting each cluster to have
$2\ell+1$ eigenvalues. If such clusters cannot be found, we decrease
$n_{cl}$ until each cluster has an integer multiple of $2\ell+1$
eigenvalues while $n_{cl}\ge n/2$. It is possible to not find clusters
that match our requirements, because of many reasons, for instance,
the number of Fourier harmonics $\left(\ell\right)$ is too low, the
actual spectral circles are too close to each other to resolve them
numerically or there are multiplicities in the eigenvalues.

From each cluster, we choose the eigenvector that has the smallest
dominant harmonic to form our vector bundle. Let us denote the eigenvalue
and eigenvector pair of index $p$ by $\lambda_{p}$ and $u_{ijp}$,
such that $\boldsymbol{u}_{p}\left(\vartheta_{j}\right)=\boldsymbol{e}_{i}u_{ijp}$.
The number of clusters is $m$ and the set of indices that belong
to cluster $1\le k\le m$ is $\mathrm{cl}\left(k\right)$. The Fourier
components of the eigenvectors are $\tilde{u}_{ilp}=\sum_{j=1}^{2\ell+1}\mathrm{e}^{-il\vartheta_{j}}u_{ijp}$,
where $-\ell\le l\le\ell$. In this notation, the index of the representative
eigenvector and eigenvalue pair is calculated as
\begin{equation}
p_{k}=\arg\min_{p\in\mathrm{cl}\left(k\right)}\left(\arg\max_{l}\sum_{i}\left|\tilde{u}_{ilp}\right|^{2}\right).\label{eq:BUND-sel}
\end{equation}

If cluster $k$ has only $2\ell+1$ eigenvalues, we expect that after
multiplying the eigenvector with a scalar from the complex unit circle,
it becomes a real-valued vector and we set
\begin{equation}
U_{1ij}^{d,k}=u_{ijp_{k}},\;\Lambda^{k}=\lambda_{p_{k}}.\label{eq:BUND-real}
\end{equation}
If this is not the case, we reject the calculation. In the common
case when cluster $k$ contains $2\left(2\ell+1\right)$ eigenvalues,
the real part and the imaginary part of the selected eigenvector represents
the invariant bundle
\begin{equation}
U_{1ij}^{d,k}=\Re u_{ijp_{k}},\;U_{2ij}^{d,k}=\Im u_{ijp_{k}},\;\boldsymbol{\Lambda}^{k}=\begin{pmatrix}\Re\lambda_{p_{k}} & -\Im\lambda_{p_{k}}\\
\Im\lambda_{p_{k}} & \Re\lambda_{p_{k}}
\end{pmatrix}.\label{eq:BUND-complex}
\end{equation}
Given index set $\mathcal{I}$, defined by \eqref{eq:ED-index-set}
and $\mathcal{I}^{c}=\left\{ 1,\ldots,m\right\} \setminus\mathcal{I}=\left\{ j_{1},\ldots,j_{m-\#\mathcal{I}}\right\} $,
we set 
\begin{equation}
\boldsymbol{U}^{d}=\begin{pmatrix}\boldsymbol{U}^{d,i_{1}}\\
\vdots\\
\boldsymbol{U}^{d,i_{\#\mathcal{I}}}
\end{pmatrix},\;\boldsymbol{V}^{d}=\begin{pmatrix}\boldsymbol{U}^{d,j_{0}}\\
\vdots\\
\boldsymbol{U}^{d,j_{m-\#\mathcal{I}}}
\end{pmatrix},\label{eq:BUND-zero-left}
\end{equation}
and 
\[
\boldsymbol{R}^{d}=\begin{pmatrix}\boldsymbol{\Lambda}^{i_{1}} &  & \boldsymbol{0}\\
 & \ddots\\
\boldsymbol{0} &  & \boldsymbol{\Lambda}^{i_{\#\mathcal{I}}}
\end{pmatrix},\;\boldsymbol{S}^{d}=\begin{pmatrix}\boldsymbol{\Lambda}^{j_{1}} &  & \boldsymbol{0}\\
 & \ddots\\
\boldsymbol{0} &  & \boldsymbol{\Lambda}^{j_{m-\#\mathcal{I}}}
\end{pmatrix}.
\]
For numerical reasons, we transform $\boldsymbol{U}^{d}$, $\boldsymbol{V}^{d}$,
so that at each grid point $\vartheta_{l}$ they are orthogonal matrices.
The transformation is carried out via singular value decomposition.
We denote $\boldsymbol{U}_{l}^{d}=\boldsymbol{U}^{d}\left(\vartheta_{l}\right)$
and its singular value decomposition by $\boldsymbol{U}_{l}^{d}=\boldsymbol{G}_{l}^{T}\boldsymbol{\Sigma}_{l}\boldsymbol{H}_{l}$
and define $\boldsymbol{U}_{l}^{\boxempty}=\boldsymbol{G}_{l}^{T}\boldsymbol{H}_{l}$.
The transformation that brings back $\boldsymbol{U}_{l}^{\boxempty}$
into $\boldsymbol{U}_{l}^{d}$ is $\boldsymbol{T}_{l}=\boldsymbol{G}_{l}^{T}\boldsymbol{\Sigma}_{l}\boldsymbol{G}_{l}$,
such that $\boldsymbol{U}_{l}^{d}=\boldsymbol{T}_{l}\boldsymbol{U}_{l}^{\boxempty}$.
The linear invariance equation \eqref{eq:ED-left-bundle} now reads
\begin{align*}
\boldsymbol{R}^{d}\boldsymbol{T}\left(\theta\right)\boldsymbol{U}^{\boxempty}\left(\theta\right) & =\boldsymbol{T}\left(\theta+\omega\right)\boldsymbol{U}^{\boxempty}\left(\theta+\omega\right)\boldsymbol{A}\left(\theta\right)\\
\boldsymbol{T}^{-1}\left(\theta+\omega\right)\boldsymbol{R}^{d}\boldsymbol{T}\left(\theta\right)\boldsymbol{U}^{\boxempty}\left(\theta\right) & =\boldsymbol{U}^{\perp}\left(\theta+\omega\right)\boldsymbol{A}\left(\theta\right)
\end{align*}
and therefore $\boldsymbol{U}^{\boxempty}$ and 
\begin{equation}
\boldsymbol{R}^{\boxempty}\left(\theta\right)=\boldsymbol{T}^{-1}\left(\theta+\omega\right)\boldsymbol{R}^{d}\boldsymbol{T}\left(\theta\right)\label{eq:BUND-Rsquare}
\end{equation}
also satisfy the invariance equation. Using index notation, the transformed
linear map becomes $R_{ijl}^{\boxempty}=T_{ipr}^{-1}\mathbb{S}_{r\underline{l}}^{-\omega}R_{pq}^{d}T_{qj\underline{l}}$.
We carry out the same transformation on $\boldsymbol{V}^{d}$ and
$\boldsymbol{S}^{d}$, to end up with a point-wise orthogonal $\boldsymbol{V}^{\boxempty}$
and $\theta$-dependent $\boldsymbol{S}^{\boxempty}$.

In summary, $\boldsymbol{U}^{\boxempty}$ and $\boldsymbol{V}^{\boxempty}$
define a new coordinate system about the invariant torus in which
the linear dynamics is block-diagonal, but not autonomous. In what
follows, the data will be transformed into this coordinate systems,
so that the functional representation of the identified invariant
foliation is relatively simple.

\subsection{\label{subsec:FOIL-FIT}Finding the invariant foliation}

We now describe how to find the invariant foliation and a more accurate
invariant torus relative to the approximate linear coordinate system
calculated in section \ref{subsec:Bundles}. Here we solve the invariance
equations \eqref{eq:FOIL-invariance} and \eqref{eq:FOIL-compl-inv}
together as they become coupled through scaling. 

For numerical reasons, we transform our data into the time-dependent
coordinate system calculated in section \ref{subsec:Bundles}, which
becomes
\begin{equation}
\begin{array}{rlcrl}
x_{i}^{k\parallel} & =U_{ijq}^{\boxempty}t_{q}^{k}\left(x_{j}^{k}-K_{jl}t_{l}^{k}\right), &  & y_{i}^{k\parallel} & =U_{ijq}^{\boxempty}t_{q}^{k}\left(y_{j}^{k}-K_{jl}\mathbb{S}_{jp}^{-\omega}t_{p}^{k}\right),\\
x_{i}^{k\perp} & =V_{ijq}^{\boxempty}t_{q}^{k}\left(x_{j}^{k}-K_{jl}t_{l}^{k}\right), &  & y_{i}^{k\perp} & =V_{ijq}^{\boxempty}t_{q}^{k}\left(y_{j}^{k}-K_{jl}\mathbb{S}_{jp}^{-\omega}t_{p}^{k}\right).
\end{array}\label{eq:FOIL-datatran}
\end{equation}
The data points are accompanied by the interpolation weights on the
torus, which are
\begin{equation}
t_{j}^{k}=\gamma\left(\theta^{k}-\vartheta_{j}\right),\;t_{j}^{\omega k}=\gamma\left(\theta^{k}+\omega-\vartheta_{j}\right)=\mathbb{S}_{jp}^{-\omega}t_{p}^{k}.\label{eq:FOIL-theta}
\end{equation}
The numerical representation of the encoders present in the invariance
equations \eqref{eq:FOIL-invariance} and \eqref{eq:FOIL-compl-inv}
are given by 
\begin{align*}
\boldsymbol{U}\left(\boldsymbol{x}^{\parallel},\boldsymbol{x}^{\perp},\theta\right) & =\boldsymbol{U}^{c}\left(\theta\right)+\boldsymbol{x}^{\parallel}+\boldsymbol{U}^{l}\left(\theta\right)\boldsymbol{x}^{\perp}+\boldsymbol{U}^{nl}\left(\boldsymbol{x}^{\parallel},\boldsymbol{x}^{\perp},\theta\right),\\
\boldsymbol{V}\left(\boldsymbol{x}^{U},\boldsymbol{x}^{V},\theta\right) & =\boldsymbol{V}^{c}\left(\theta\right)+\boldsymbol{x}^{\perp}+\boldsymbol{V}^{l}\left(\theta\right)\boldsymbol{x}^{\parallel}+\boldsymbol{V}^{nl}\left(\boldsymbol{x}^{\parallel},\boldsymbol{x}^{\perp},\theta\right),
\end{align*}
where 
\begin{gather*}
\boldsymbol{U}^{c}:\mathbb{T}\to Z,\;\boldsymbol{U}^{l}:\mathbb{T}\to L\left(Z^{c},Z\right),\;\boldsymbol{U}^{nl}:Z\times Z^{c}\times\mathbb{T}\to Z,\\
\boldsymbol{V}^{c}:\mathbb{T}\to Z^{c},\;\boldsymbol{V}^{l}:\mathbb{T}\to L\left(Z,Z^{c}\right),\;\boldsymbol{V}^{nl}:Z\times Z^{c}\times\mathbb{T}\to Z^{c}.
\end{gather*}
To make the representation of any given pair of foliations unique,
we normalise the encoders by applying the constraints 
\begin{equation}
\boldsymbol{U}^{nl}\left(\boldsymbol{0},\boldsymbol{x}^{\perp},\theta\right)=\boldsymbol{0}\;\text{and}\;\boldsymbol{V}^{nl}\left(\boldsymbol{x}^{\parallel},\boldsymbol{0},\theta\right)=\boldsymbol{0},\label{eq:FOIL-constraints}
\end{equation}
which is the numerical equivalent of \eqref{eq:FOIL-graph-constr}.
The simplicity of the constraints \eqref{eq:FOIL-constraints} is
a consequence of our coordinate transformation \eqref{eq:FOIL-datatran}.
As discussed in section \ref{subsec:Foliations}, the constraints
\eqref{eq:FOIL-constraints} need not be accurate, in fact there are
many other valid constraints. However, the simplicity of \eqref{eq:FOIL-constraints}
allows us to encode $\boldsymbol{U}^{nl}$ and $\boldsymbol{V}^{nl}$
in a straightforward manner. For example, if $\boldsymbol{U}^{nl}$
is a polynomial, we drop all monomials that only contain components
of $\boldsymbol{x}^{\perp}$. Let us denote $\boldsymbol{x}=\left(\boldsymbol{x}^{\parallel},\boldsymbol{x}^{\perp}\right)$,
then using a tensor product notation we have
\begin{equation}
\boldsymbol{U}^{nl}\left(\boldsymbol{x}^{k\parallel},\boldsymbol{x}^{k\perp},\theta^{k}\right)=\sum_{d=2}^{\sigma}\boldsymbol{U}_{d}^{nl}\boldsymbol{t}^{k}\otimes\boldsymbol{x}^{\parallel}\otimes\boldsymbol{x}^{\otimes d-1},\label{eq:FOIL-Unl}
\end{equation}
which can represent any order-$\sigma$ polynomial that satisfy \eqref{eq:FOIL-constraints}.
The representation \eqref{eq:FOIL-Unl} can either use dense polynomials
or polynomials with compressed tensor coefficients as described in
\cite{Szalai2023Fol}.

For the conjugate dynamics represented by function $\boldsymbol{R}$
in the invariance equations \eqref{eq:FOIL-invariance}, we use a
dense polynomial, because it is a low-dimensional function. For the
conjugate dynamics of the complementary foliation $\boldsymbol{S}$,
we either use a dense polynomial if we are dealing with a low-dimensional
system, or a matrix in high dimensions. In fact, function $\boldsymbol{S}$
can be linear if we consider a small neighbourhood of the invariant
manifold defined by 
\begin{equation}
\mathcal{M}=\left\{ \left(\boldsymbol{x}^{\parallel},\boldsymbol{x}^{\perp},\theta\right)\in X:\boldsymbol{V}\left(\boldsymbol{x}^{\parallel},\boldsymbol{x}^{\perp},\theta\right)=\boldsymbol{0}\right\} .\label{eq:MANIF-implicit}
\end{equation}
This is because in the neighbourhood of $\mathcal{M}$, the magnitude
of $\boldsymbol{V}$ is small. During the solution process of the
invariance equation \eqref{eq:FOIL-compl-inv}, the closer we get
to the accurate solution, the data can be filtered, so that the value
of $\boldsymbol{V}$ over the remaining data set becomes small, as
long as sufficient amount of data remains.

We define the relative error of the invariance equation by 
\begin{equation}
E_{\mathit{rel}}=\frac{\left\Vert \boldsymbol{R}\left(\boldsymbol{U}\left(\boldsymbol{x}^{\parallel},\boldsymbol{x}^{\perp},\theta\right),\theta\right)-\boldsymbol{U}\left(\boldsymbol{y}^{\parallel},\boldsymbol{y}^{\perp},\theta+\omega\right)\right\Vert }{\left\Vert \left(\boldsymbol{x}^{\parallel}-\boldsymbol{K}^{\parallel}\left(\theta\right),\boldsymbol{x}^{\perp}-\boldsymbol{K}^{\perp}\left(\theta\right)\right)\right\Vert },\label{eq:FOIL-Erel}
\end{equation}
where $\boldsymbol{K}^{\parallel}$ and $\boldsymbol{K}^{\perp}$
represent the torus in the transformed coordinate system. The torus
representation is the solution of 
\[
\boldsymbol{U}\left(\boldsymbol{K}^{\parallel}\left(\theta\right),\boldsymbol{K}^{\perp}\left(\theta\right),\theta\right)=\boldsymbol{0},\;\boldsymbol{V}\left(\boldsymbol{K}^{\parallel}\left(\theta\right),\boldsymbol{K}^{\perp}\left(\theta\right),\theta\right)=\boldsymbol{0}.
\]
In order to make sure that the relative error is nearly uniform we
scale the invariance equations \eqref{eq:FOIL-invariance} and \eqref{eq:FOIL-compl-inv}
similar to how $E_{\mathit{rel}}$ is calculated. Let use define
\[
\delta^{k}=1+\frac{1}{\epsilon^{2}+\left\Vert \boldsymbol{x}^{k\parallel}-\boldsymbol{K}^{\parallel}\boldsymbol{t}^{k}\right\Vert ^{2}+\left\Vert \boldsymbol{x}^{k\perp}-\boldsymbol{K}^{\perp}\boldsymbol{t}^{k}\right\Vert ^{2}}
\]
and use the loss function
\begin{equation}
L=\sum_{k=1}^{N}\delta^{k}\left(\left\Vert \boldsymbol{R}\circ\boldsymbol{U}-\boldsymbol{U}^{\omega}\right\Vert ^{2}+\left\Vert \boldsymbol{S}\circ\boldsymbol{V}-\boldsymbol{V}^{\omega}\right\Vert ^{2}\right).\label{eq:CF-loss}
\end{equation}
where 
\begin{align*}
\boldsymbol{U} & :=\boldsymbol{U}\left(\boldsymbol{x}^{k\parallel},\boldsymbol{x}^{k\perp},\boldsymbol{t}^{k}\right), & \boldsymbol{R}\circ\boldsymbol{U} & :=\boldsymbol{R}\left(\boldsymbol{U}\left(\boldsymbol{x}^{k\parallel},\boldsymbol{x}^{k\perp},\boldsymbol{t}^{k}\right),\boldsymbol{t}^{k}\right), & \boldsymbol{U}^{\omega} & :=\boldsymbol{U}\left(\boldsymbol{y}^{k\parallel},\boldsymbol{y}^{k\perp},\boldsymbol{t}^{\omega k}\right),\\
\boldsymbol{V} & :=\boldsymbol{V}\left(\boldsymbol{x}^{k\parallel},\boldsymbol{x}^{k\perp},\boldsymbol{t}^{k}\right), & \boldsymbol{S}\circ\boldsymbol{V} & :=\boldsymbol{S}\left(\boldsymbol{V}\left(\boldsymbol{x}^{k\parallel},\boldsymbol{x}^{k\perp},\boldsymbol{t}^{k}\right),\boldsymbol{t}^{k}\right), & \boldsymbol{V}^{\omega} & :=\boldsymbol{V}\left(\boldsymbol{y}^{k\parallel},\boldsymbol{y}^{k\perp},\boldsymbol{t}^{\omega k}\right)
\end{align*}
with data coming from \eqref{eq:FOIL-datatran} and \eqref{eq:FOIL-theta}.
The initial values of the parameters are such that 
\begin{align*}
\boldsymbol{R}\left(\boldsymbol{z},\theta\right) & =\boldsymbol{R}^{1}\left(\theta\right)\boldsymbol{z},\boldsymbol{U}^{c}\left(\theta\right)=\boldsymbol{0},\boldsymbol{U}^{l}\left(\theta\right)=\boldsymbol{0},\boldsymbol{U}^{nl}\left(\boldsymbol{x}^{\perp},\theta\right)=\boldsymbol{0},\\
\boldsymbol{S}\left(\boldsymbol{z},\theta\right) & =\boldsymbol{S}^{1}\left(\theta\right)\boldsymbol{z},\boldsymbol{V}^{c}\left(\theta\right)=\boldsymbol{0},\boldsymbol{V}^{l}\left(\theta\right)=\boldsymbol{0},\boldsymbol{V}^{nl}\left(\boldsymbol{x}^{\perp},\theta\right)=\boldsymbol{0}
\end{align*}
with values from \eqref{eq:BUND-Rsquare}. The optimisation algorithm
uses a batch coordinate descent technique \cite{GaussSouthwell2015},
because in each individual component of the functional representation,
the problem is nearly linear, while considering all parameters at
the same time would be a more difficult problem. The scaling factors
$\delta^{k}$ and hence the invariant torus $\boldsymbol{K}^{\parallel}$
and $\boldsymbol{K}^{\perp}$ are updated after the optimisation step
for a component has concluded. This same technique was used for autonomous
systems in \cite{Szalai2023Fol}.

\subsection{\label{subsec:Recovery}Recovering the invariant manifold}

We now recover the invariant manifold from the two invariant foliations.
The invariant manifold is defined by \eqref{eq:MANIF-implicit}, which
does not stipulate how it is parameterised. We however want the invariant
manifold to have the same parametrisation as the invariant foliation
defined by encoder $\boldsymbol{U}$, which means that the conjugate
dynamics defined by $\boldsymbol{R}$ will be the same for the invariant
manifold. We now decompose the decoder of the invariant manifold into
two components in $X=Z\oplus Z^{c}$ and denote it by $\boldsymbol{W}\left(\boldsymbol{z},\theta\right)=\left(\boldsymbol{W}^{\parallel}\left(\boldsymbol{z},\theta\right),\boldsymbol{W}^{\perp}\left(\boldsymbol{z},\theta\right)\right)$.
The unknown decoder functions are found by solving the equations
\begin{equation}
\left.\begin{array}{rl}
\boldsymbol{U}\left(\boldsymbol{W}^{\parallel}\left(\boldsymbol{z},\theta\right),\boldsymbol{W}^{\perp}\left(\boldsymbol{z},\theta\right),\theta\right) & =\boldsymbol{z}\\
\boldsymbol{V}\left(\boldsymbol{W}^{\parallel}\left(\boldsymbol{z},\theta\right),\boldsymbol{W}^{\perp}\left(\boldsymbol{z},\theta\right),\theta\right) & =\boldsymbol{0}
\end{array}\right\} .\label{eq:MANIF-comp-implicit}
\end{equation}
Equations \eqref{eq:MANIF-comp-implicit} can be re-written to the
following form
\[
\begin{pmatrix}\boldsymbol{I} & \boldsymbol{U}^{l}\left(\theta\right)\\
\boldsymbol{V}^{l}\left(\theta\right) & \boldsymbol{I}
\end{pmatrix}\begin{pmatrix}\boldsymbol{W}^{\parallel}\\
\boldsymbol{W}^{\perp}
\end{pmatrix}=\begin{pmatrix}\boldsymbol{z}-\boldsymbol{U}^{c}\left(\theta\right)-\boldsymbol{U}^{nl}\left(\boldsymbol{W}^{\perp},\theta\right)\\
\boldsymbol{0}-\boldsymbol{V}^{c}\left(\theta\right)-\boldsymbol{V}^{nl}\left(\boldsymbol{W}^{\parallel},\theta\right)
\end{pmatrix},
\]
which can be solved by the iteration
\[
\begin{pmatrix}\boldsymbol{W}^{\parallel}\\
\boldsymbol{W}^{\perp}
\end{pmatrix}=\begin{pmatrix}\boldsymbol{I} & \boldsymbol{U}^{l}\left(\theta\right)\\
\boldsymbol{V}^{l}\left(\theta\right) & \boldsymbol{I}
\end{pmatrix}^{-1}\begin{pmatrix}\boldsymbol{z}-\boldsymbol{U}^{c}\left(\theta\right)-\boldsymbol{U}^{nl}\left(\boldsymbol{W}^{\perp},\theta\right)\\
\boldsymbol{0}-\boldsymbol{V}^{c}\left(\theta\right)-\boldsymbol{V}^{nl}\left(\boldsymbol{W}^{\parallel},\theta\right)
\end{pmatrix}.
\]
Given that our functions are represented by polynomials, the iteration
will converge from the initial guess $\boldsymbol{W}^{\parallel}=\boldsymbol{0}$,
$\boldsymbol{W}^{\perp}=\boldsymbol{0}$ to a polynomial of the same
order as $\boldsymbol{U}^{nl}$, $\boldsymbol{V}^{nl}$ in the same
number of steps as the highest polynomial order of $\boldsymbol{U}^{nl}$,
$\boldsymbol{V}^{nl}$.

\subsection{\label{subsec:NFORM}Normal form transformation of the conjugate
dynamics}

In order to extract information from the conjugate dynamics $\boldsymbol{R}$,
we transform it into its simplest form. This is the place where we
account for internal and parametric resonances. The transformation
uses the invariant manifold style to create a decoder. This style
is useful when we have an invariant manifold, for example as a result
of the recovery procedure in section \ref{subsec:Recovery}.

The invariance equation we use is given by 
\begin{equation}
\boldsymbol{T}\left(\breve{\boldsymbol{R}}\left(\boldsymbol{z},\theta\right),\theta+\omega\right)-\boldsymbol{R}^{d}\left(\boldsymbol{T}\left(\boldsymbol{z}\right),\theta\right)=\boldsymbol{0},\label{eq:NORMF-invar}
\end{equation}
where $\boldsymbol{T}:Z\times\mathbb{T}\to Z$ is the transformation
and $\breve{\boldsymbol{R}}$ is the normal form. The map $\boldsymbol{R}^{d}$
with diagonalised linear part is calculated by 
\[
\boldsymbol{R}^{d}\left(\boldsymbol{z},\theta\right)=\boldsymbol{U}^{d}\left(\theta+\omega\right)\boldsymbol{R}\left(\left(\boldsymbol{U}^{d}\left(\theta\right)\right)^{-1}\boldsymbol{z},\theta\right),
\]
where $\boldsymbol{R}$ is the conjugate dynamics that we obtained
from our invariant foliation also associated with the recovered invariant
manifold $\mathcal{M}$ in equation \eqref{eq:MANIF-implicit}. The
linear transformation $\boldsymbol{U}^{d}$ is the full invariant
bundle decomposition of $D_{1}\boldsymbol{R}\left(\boldsymbol{0},\cdot\right)$
calculated using the methods discussed in section \ref{subsec:Bundles},
and given by formula \eqref{eq:BUND-zero-left} for an index set $\mathcal{I}$
that covers the full dichotomy spectrum of $D_{1}\boldsymbol{R}\left(\boldsymbol{0},\cdot\right)$.
After the transformation, we have $\boldsymbol{R}^{d}\left(\boldsymbol{z},\theta\right)=\boldsymbol{\Lambda}\boldsymbol{z}+\boldsymbol{N}\left(\boldsymbol{z},\theta\right)$,
where $\boldsymbol{\Lambda}$ is diagonal, with possibly complex entries
and $\boldsymbol{N}\left(\boldsymbol{z},\theta\right)=\mathcal{O}\left(\left|\boldsymbol{z}\right|^{2}\right)$.

We represent the solution of \eqref{eq:NORMF-invar} by trigonometric
power series
\begin{equation}
\left.\begin{array}{rl}
\boldsymbol{T}\left(\boldsymbol{z},\theta\right) & =\sum_{j=1}^{\sigma}\boldsymbol{T}^{j}\left(\theta\right)\boldsymbol{z}^{\otimes j}\\
\breve{\boldsymbol{R}}\left(\boldsymbol{z},\theta\right) & =\sum_{j=1}^{\sigma}\breve{\boldsymbol{R}}^{j}\left(\theta\right)\boldsymbol{z}^{\otimes j}
\end{array}\right\} ,\label{eq:NFORM-ansatz}
\end{equation}
where
\begin{equation}
\left.\begin{array}{rl}
\boldsymbol{T}^{j}\left(\boldsymbol{\theta}\right)\boldsymbol{z}^{\otimes j} & =\sum_{i_{0}\cdots i_{j},k}\boldsymbol{e}_{i_{0}}T_{i_{0}i_{1}\cdots i_{j}}^{j,k}\mathrm{e}^{ik\theta}z_{i_{1}}\cdots z_{i_{j}},\\
\breve{\boldsymbol{R}}^{j}\left(\boldsymbol{\theta}\right)\boldsymbol{z}^{\otimes j} & =\sum_{i_{0}\cdots i_{j},k}\boldsymbol{e}_{i_{0}}\breve{R}_{i_{0}i_{1}\cdots i_{j}}^{j,k}\mathrm{e}^{ik\theta}z_{i_{1}}\cdots z_{i_{j}},
\end{array}\right\} \label{eq:NFORM-trafo}
\end{equation}
and $\boldsymbol{e}_{i}$, $i\in\left\{ 1,\ldots,\dim Z\right\} $
is our orthonormal basis in $Z$. Substituting the ansatz \eqref{eq:NFORM-ansatz}
into \eqref{eq:NORMF-invar} and separating the $j$th order terms
leads to the homological equation

\begin{equation}
\boldsymbol{T}^{1}\left(\boldsymbol{\theta}+\boldsymbol{\omega}\right)\breve{\boldsymbol{R}}^{j}\left(\boldsymbol{\theta}\right)\boldsymbol{z}^{\otimes j}+\boldsymbol{T}^{j}\left(\boldsymbol{\theta}+\boldsymbol{\omega}\right)\left(\boldsymbol{\Lambda}\boldsymbol{x}\right)^{\otimes j}-\boldsymbol{\Lambda}\boldsymbol{T}^{j}\left(\boldsymbol{\theta}\right)\boldsymbol{z}^{\otimes j}=\boldsymbol{\Gamma}^{j}\left(\boldsymbol{\theta}\right)\boldsymbol{z}^{\otimes j},\label{eq:NFORM-homological}
\end{equation}
where $\boldsymbol{\Gamma}^{j}\left(\boldsymbol{\theta}\right)$ are
terms composed of $\boldsymbol{R}^{d}$ and $\boldsymbol{T}$ of order
lower than $j$. Successively solving equation \eqref{eq:NFORM-homological}
leads to a normal form.

For the linear terms, the solution of equation \eqref{eq:NFORM-homological}
can be chosen as $\boldsymbol{T}^{1}\left(\boldsymbol{\theta}\right)=\boldsymbol{I}$
and $\breve{\boldsymbol{R}}^{1}\left(\boldsymbol{\theta}\right)=\boldsymbol{\Lambda}$.
The nonlinear terms are determined from equation
\[
\breve{R}_{i_{0}i_{1}\cdots i_{j}}^{j,k}+\lambda_{i_{1}}\cdots\lambda_{i_{j}}\mathrm{e}^{ik\omega}T_{i_{0}i_{1}\cdots i_{j}}^{j,k}-\lambda_{i_{0}}T_{i_{0}i_{1}\cdots i_{j}}^{j,k}=\Gamma_{i_{0}i_{1}\cdots i_{j}}^{j,k},
\]
which has the solution
\begin{align}
T_{i_{0}i_{1}\cdots i_{j}}^{j,k} & =\frac{1}{\lambda_{i_{1}}\cdots\lambda_{i_{j}}\mathrm{e}^{ik\omega}-\lambda_{i_{0}}}\Gamma_{i_{0}i_{1}\cdots i_{j}}^{j,k} & \breve{R}_{i_{0}i_{1}\cdots i_{j}}^{j,k} & =0\;\text{or}\label{eq:NFORM-std}\\
T_{i_{0}i_{1}\cdots i_{j}}^{j,k} & =0 & \breve{R}_{i_{0}i_{1}\cdots i_{j}}^{j,k} & =\Gamma_{i_{0}i_{1}\cdots i_{j}}^{j,k}.\label{eq:NFORM-res}
\end{align}
We call the case when $\lambda_{i_{1}}\cdots\lambda_{i_{j}}\mathrm{e}^{ik\omega}-\lambda_{i_{0}}\approx0$
an internal resonance, in which case the solution \eqref{eq:NFORM-res}
is used. As per our assumptions we only allow internal resonance for
$k=0$, which leads to an autonomous normal form.

Putting together all transformations, we define 
\[
\breve{\boldsymbol{W}}\left(\boldsymbol{z},\theta\right)=\boldsymbol{W}\left(\left(\boldsymbol{U}^{d}\left(\theta\right)\right)^{-1}\boldsymbol{T}\left(\boldsymbol{z},\theta\right),\theta\right).
\]
We find that the normal form $\breve{\boldsymbol{R}}$ satisfies the
manifold invariance equation
\[
\breve{\boldsymbol{W}}\left(\breve{\boldsymbol{R}}\left(\boldsymbol{z},\theta\right),\theta+\omega\right)=\breve{\boldsymbol{F}}\left(\breve{\boldsymbol{W}}\left(\boldsymbol{z},\theta\right),\theta\right),
\]
where $\breve{\boldsymbol{F}}$ is the unknown system in the frame
of the approximate vector bundles defined by $\boldsymbol{U}^{1}$,
$\boldsymbol{V}^{1}$.

\subsection{\label{subsec:FREQ-DAMP}Frequencies and damping ratios}

We now assume that the normal form transformation in section \ref{subsec:NFORM}
has led to an autonomous system, that is $\breve{\boldsymbol{R}}\left(\boldsymbol{z},\theta\right)=\breve{\boldsymbol{R}}\left(\boldsymbol{z}\right)$.
When we calculate the invariant foliation for a vector bundle associated
with a complex conjugate pair of eigenvalues, the conjugate dynamics
can also be written as a pair of complex conjugate functions \cite{wiggins2003introduction}
of the complex variable $z$ in the form of 
\[
\breve{\boldsymbol{R}}\left(\boldsymbol{z}\right)=\begin{pmatrix}s\left(z,\overline{z}\right)\\
\overline{s}\left(z,\overline{z}\right)
\end{pmatrix},
\]
where $\overline{\;}$ means complex conjugate. We now define 
\begin{align*}
\widehat{\boldsymbol{W}}\left(r,\beta,\theta\right) & =\breve{\boldsymbol{W}}\left(r\mathrm{e}^{i\beta},r\mathrm{e}^{-i\beta},\theta\right),\\
R\left(r\right) & =\left|s\left(r\mathrm{e}^{i\beta},r\mathrm{e}^{-i\beta}\right)\right|\\
T\left(r\right) & =\arg\,\mathrm{e}^{-i\beta}s\left(r\mathrm{e}^{i\beta},r\mathrm{e}^{-i\beta}\right)
\end{align*}
and the invariance equation becomes
\[
\widehat{\boldsymbol{W}}\left(R\left(r\right),\beta+T\left(r\right),\theta+\omega\right)=\boldsymbol{F}\left(\widehat{\boldsymbol{W}}\left(r,\beta,\theta\right),\theta\right).
\]
Functions $R$ and $T$ are independent of $\beta$ because we have
eliminated this dependence during the normal-form transformation by
only keeping near resonant terms.

The functions $R$ and $T$ also relate to the damping and the frequency
of the dynamics, in fact if we were to consider $Z$ as an Euclidean
frame, the frequency of the dynamics would be $\omega\left(r\right)=T\left(r\right)/\Delta t$
and the damping would be $\xi\left(r\right)=\log\left(r^{-1}R\left(r\right)\right)/T\left(r\right)$.
However the decoder $\widehat{\boldsymbol{W}}$ defines a nonlinear
frame in $X$ and therefore we need to calculate the damping and the
frequency by taking into account the nonlinearity of $\widehat{\boldsymbol{W}}$
.

There are two factors that makes the calculations inaccurate: the
amplitude predicted by $\widehat{\boldsymbol{W}}$ does not increase
linearly with parameter $r$ and the phase angles of two $d+1$ dimensional
tori for two fixed values of $r$ do not align in $X$ and therefore
when a solution decays, in each cycle a phase shift occurs relative
to what $T\left(r\right)$ predicts, which leads to a frequency miscalculation.
Here we fix both of these discrepancies.

To carry out the correction we use the transformation $r=\rho\left(\hat{r}\right)$
and $\beta=\hat{\beta}+\alpha\left(\rho\left(\hat{r}\right)\right)$,
where $\rho:\left[0,\infty\right)\to\left[0,\infty\right)$ and $\alpha:\left[0,\infty\right)\to\mathbb{R}$,
and define
\begin{equation}
\tilde{\boldsymbol{W}}\left(r,\beta,\theta\right)=\widehat{\boldsymbol{W}}\left(\rho\left(r\right),\beta+\alpha\left(\rho\left(r\right)\right),\theta\right).\label{eq:FREQ-reparam}
\end{equation}
This transformation re-scales the amplitude and introduces an amplitude
dependent phase shift. We do not shift the phase of the forcing, because
that would make the forcing frequency dependent on the vibration amplitude.
The two-dimensional tori parametrised by the amplitude parameter are
\[
\mathcal{T}_{r}=\left\{ \tilde{\boldsymbol{W}}\left(r,\beta,\theta\right):\beta,\theta\in[0,2\pi)\right\} .
\]
To make sure that the amplitude of the torus grows linearly with $r$,
we impose the constraint
\begin{equation}
\left(2\pi\right)^{-2}\int_{\mathbb{T}^{2}}\left|\widehat{\boldsymbol{W}}\left(\rho\left(r\right),\beta,\theta\right)\right|^{2}\mathrm{d}\beta\mathrm{d}\theta=r^{2}.\label{eq:FREQ-ampl}
\end{equation}
We also impose zero phase shift between nearby tori $\mathcal{T}_{r}$
and $\mathcal{T}_{r+\epsilon}$ using the constraint
\begin{equation}
\arg\min_{\delta}\int_{\mathbb{T}^{2}}\left|\widehat{\boldsymbol{W}}\left(\rho\left(r+\epsilon\right),\beta+\alpha\left(\rho\left(r+\epsilon\right)\right)+\delta,\theta\right)-\widehat{\boldsymbol{W}}\left(\rho\left(r\right),\beta+\alpha\left(\rho\left(r\right)\right),\theta\right)\right|^{2}\mathrm{d}\beta\mathrm{d}\theta\Big\vert_{\delta=0}=0,\label{eq:FREQ-phase}
\end{equation}
which stipulates that the distance between the points on the two tori
corresponding at the same parameters is minimal on average at zero
additional phase shift $\delta=0$. To satisfy constraint \eqref{eq:FREQ-ampl},
we define the function 
\[
\kappa\left(\rho\right)=\left(2\pi\right)^{-\left(d+1\right)/2}\sqrt{\int_{\mathbb{T}^{2}}\left|\widehat{\boldsymbol{W}}\left(\rho\left(r\right),\beta,\theta\right)\right|^{2}\mathrm{d}\beta\mathrm{d}\theta},
\]
and set
\begin{align*}
\rho\left(r\right) & =\kappa^{-1}\left(r\right).
\end{align*}
The necessary condition to find a minimum of \eqref{eq:FREQ-phase}
is a vanishing gradient, hence we calculate the gradient of
\[
L\left(\delta\right)=\frac{1}{2}\int_{\mathbb{T}^{2}}\left|\widehat{\boldsymbol{W}}\left(\rho\left(r+\epsilon\right),\beta+\alpha\left(\rho\left(r+\epsilon\right)\right)+\delta,\theta\right)-\widehat{\boldsymbol{W}}\left(\rho\left(r\right),\beta+\alpha\left(\rho\left(r\right)\right),\theta\right)\right|^{2}\mathrm{d}\beta\mathrm{d}\theta,
\]
at $\delta=0$, which leads to the equation of the necessary condition
\begin{multline}
\int_{\mathbb{T}^{2}}\Bigl\langle\widehat{\boldsymbol{W}}\left(\rho\left(r+\epsilon\right),\beta+\alpha\left(\rho\left(r+\epsilon\right)\right),\theta\right)-\widehat{\boldsymbol{W}}\left(\rho\left(r\right),\beta+\alpha\left(\rho\left(r\right)\right),\theta\right),\\
D_{2}\widehat{\boldsymbol{W}}\left(\rho\left(r+\epsilon\right),\beta+\alpha\left(\rho\left(r+\epsilon\right)\right),\theta\right)\Bigr\rangle\mathrm{d}\beta\mathrm{d}\theta=0.\label{eq:FREQ-phase-necess}
\end{multline}
Dividing equation \eqref{eq:FREQ-phase-necess} by $\epsilon$ and
letting $\epsilon\to0$, we obtain
\begin{multline*}
\lim_{\epsilon\to0}\epsilon^{-1}\left[\widehat{\boldsymbol{W}}\left(\rho\left(r+\epsilon\right),\beta+\alpha\left(\rho\left(r+\epsilon\right)\right),\theta\right)-\widehat{\boldsymbol{W}}\left(\rho\left(r\right),\beta+\alpha\left(\rho\left(r\right)\right),\theta\right)\right]\\
=D_{1}\widehat{\boldsymbol{W}}\left(\rho\left(r\right),\beta+\alpha\left(\rho\left(r\right)\right),\theta\right)\dot{\rho}\left(r\right)+D_{2}\widehat{\boldsymbol{W}}\left(\rho\left(r\right),\beta+\alpha\left(\rho\left(r\right)\right),\theta\right)\dot{\alpha}\left(\rho\left(r\right)\right)\dot{\rho}\left(r\right)=0.
\end{multline*}
Further eliminating the non-zero $\dot{\rho}$ and substituting $r\to\rho^{-1}\left(r\right)$,
we get
\[
\int_{\mathbb{T}^{2}}\left\langle D_{1}\widehat{\boldsymbol{W}}\left(r,\beta,\theta\right)+D_{2}\widehat{\boldsymbol{W}}\left(r,\beta,\theta\right)\dot{\alpha}\left(r\right),D_{2}\widehat{\boldsymbol{W}}\left(r,\beta,\theta\right)\right\rangle \mathrm{d}\beta\mathrm{d}\theta=0,
\]
whose solution for $\dot{\alpha}$ is 
\begin{equation}
\dot{\alpha}\left(r\right)=-\left[\int_{\mathbb{T}^{2}}\left\langle D_{2}\widehat{\boldsymbol{W}}\left(r,\beta,\theta\right),D_{2}\widehat{\boldsymbol{W}}\left(r,\beta,\theta\right)\right\rangle \mathrm{d}\beta\mathrm{d}\theta\right]^{-1}\int_{\mathbb{T}^{2}}\left\langle D_{1}\widehat{\boldsymbol{W}}\left(r,\beta,\theta\right),D_{2}\widehat{\boldsymbol{W}}\left(r,\beta,\theta\right)\right\rangle \mathrm{d}\beta\mathrm{d}\theta.\label{eq:FREQ-phase-DE}
\end{equation}
The expression of $\dot{\alpha}$ can be integrated with initial condition
$\alpha\left(0\right)=0$ to recover the required phase shift.

The new parametrisation as defined by equation \eqref{eq:FREQ-reparam}
satisfies the invariance equation
\[
\tilde{\boldsymbol{W}}\left(\tilde{R}\left(r\right),\beta+\tilde{T}\left(r\right),\boldsymbol{\theta}\right)=\overline{\boldsymbol{F}}\left(\tilde{\boldsymbol{W}}\left(r,\beta,\boldsymbol{\theta}\right),\theta\right),
\]
where

\begin{align*}
\tilde{R}\left(r\right) & =\rho^{-1}\left(R\left(\rho\left(r\right)\right)\right),\\
\tilde{T}\left(r\right) & =T\left(\rho\left(r\right)\right)+\alpha\left(\rho\left(r\right)\right)-\alpha\left(R\left(\rho\left(r\right)\right)\right).
\end{align*}
Finally, the instantaneous frequency and damping ratio becomes 
\begin{align*}
\omega\left(r\right) & =\tilde{T}\left(r\right)/\Delta t\\
\xi\left(r\right) & =\log\left(r^{-1}\tilde{R}\left(r\right)\right)/\tilde{T}\left(r\right),
\end{align*}
respectively, where $\Delta t$ is the sampling period. The ROM with
accurate frequencies and damping can be written in the ordinary differential
equation form
\begin{align*}
\dot{r} & =-\zeta\left(r\right)\omega\left(r\right)r,\\
\dot{\theta} & =\omega\left(r\right).
\end{align*}

\subsection{\label{subsec:ROMid}Summary of ROM identification}

The ROM identification has generally three major steps, that can be
characterised as pre-processing the data, fitting the foliation to
data and analysing the result. Each major step can be broken down
into smaller tasks as follows.
\begin{enumerate}
\item Pre-processing the data
\begin{enumerate}
\item Identification of an approximate invariant torus and the linear dynamics
about the torus. (Section \ref{subsec:LINID}.)
\item Decomposition of the linear dynamics into invariant vector bundles.
(Section \ref{subsec:Bundles}.)
\item Transforming the data into the coordinate system of the vector bundles.
(Equations \eqref{eq:FOIL-datatran} and \eqref{eq:FOIL-theta}.)
\end{enumerate}
\item Fitting two invariant foliations to the data, one that contains the
dynamics of interest and one that is possibly less accurate, but defines
the invariant manifold of the the dynamics of interest. (Section \ref{subsec:FOIL-FIT}.)
\item Post-processing the result
\begin{enumerate}
\item Recovering the invariant manifold from the foliation. (Section \ref{subsec:Recovery}.)
\item If necessary, calculate instantaneous frequencies and damping ratios
that are accurate. (Sections \ref{subsec:NFORM} and \ref{subsec:FREQ-DAMP}.)
\end{enumerate}
\end{enumerate}
In what follows we illustrate this procedure through examples.

\section{\label{sec:examples}Examples}

We will use synthetic data as the author does not have access to experimental
data of the kind necessary to demonstrate the method. In addition
to the methods discussed, we also calculate invariant manifolds directly
from vector fields. This includes finding the invariant torus and
calculating the invariant vector bundles about the torus \cite{SzalaiForcedTheory2024}.
We also use a facility in the Julia programming language that allows
Taylor expansion of solutions of differential equations through automatic
differentiation. Hence we can also calculate invariant manifolds and
invariant foliations from the generated maps without using simulation
data.

In what follows we will also display the eigenvalues \eqref{eq:BUND-tens-eigvals}
in a vector field form, that is
\[
\lambda_{\mathit{vf}}=\frac{1}{\Delta t}\log\lambda_{\mathit{map}}
\]
This helps visualisation, because eigenvalues are now appear on vertical
lines as opposed to concentric circles. We analyse the relative error
\eqref{eq:FOIL-Erel} of our calculations with respect to the invariance
equation \eqref{eq:FOIL-invariance}. To provide a basis for comparison
we solve the manifold invariance equation for the underlying ordinary
differential equation in our examples, which is
\begin{equation}
D_{1}\boldsymbol{W}\left(\boldsymbol{z},\theta\right)\boldsymbol{R}\left(\boldsymbol{z},\theta\right)+\omega_{0}D_{2}\boldsymbol{W}\left(\boldsymbol{z},\theta\right)=\boldsymbol{f}\left(\boldsymbol{W}\left(\boldsymbol{z},\theta\right),\theta\right),\label{eq:MANIF-ODE-invar}
\end{equation}
where $\boldsymbol{f}$ represents the ordinary differential equation
\begin{align*}
\dot{\boldsymbol{x}} & =\boldsymbol{f}\left(\boldsymbol{x},\theta\right),\\
\dot{\theta} & =\omega_{0}.
\end{align*}
The solution methods is similar to the normal for transformation detailed
in section \ref{subsec:NFORM}. For the direct calculation we use
\begin{equation}
E_{\mathit{rel}}\left(\boldsymbol{z},\theta\right)=\frac{\left\Vert D_{1}\boldsymbol{W}\left(\boldsymbol{z},\theta\right)\boldsymbol{R}\left(\boldsymbol{z},\theta\right)+\omega_{0}D_{2}\boldsymbol{W}\left(\boldsymbol{z},\theta\right)-\boldsymbol{f}\left(\boldsymbol{W}\left(\boldsymbol{z},\theta\right),\theta\right)\right\Vert }{\left\Vert \boldsymbol{W}\left(\boldsymbol{z},\theta\right)\right\Vert }\label{eq:E-rel-ode}
\end{equation}
for the relative error and $\left\Vert \boldsymbol{W}\left(\boldsymbol{z},\theta\right)\right\Vert $
for the amplitude.

\subsection{Shaw-Pierre oscillator}

We consider a variant of the well-studied example of reduced order
modelling that appeared in \cite{ShawPierre}. The equations of motions
are written as
\begin{equation}
\begin{array}{rl}
\dot{x}_{1} & =x_{3},\\
\dot{x}_{2} & =x_{4},\\
\dot{x}_{3} & =-cx_{3}-kx_{1}-\kappa y_{1}^{3}+k\left(y_{2}-y_{1}\right)+c\left(y_{4}-y_{3}\right)+A\cos\left(\omega t+0.1\right),\\
\dot{x}_{4} & =-cy_{4}-ky_{2}-k\left(y_{2}-y_{1}\right)-c\left(y_{4}-y_{3}\right)+A\cos\left(\omega t\right).
\end{array}\label{eq:ode-shawpierre}
\end{equation}
We choose the parameters $k=1$, $\kappa=0.2$ and $c=2^{-5}$. The
unforced natural frequencies are $\omega_{1}=1$ and $\omega_{2}=1.7314$
and the damping ratios are $\zeta_{1}=0.0156$ and $\zeta_{2}=0.0271$.
The spectral quotients are $\beth_{1}=1$ and $\beth_{2}=3$. When
we turn on the forcing at $A=0.25$, the frequencies become $\omega_{1}=1.0263$,
$\omega_{2}=1.7473$, the damping ratios become $\zeta_{1}=0.0153$,
$\zeta_{2}=0.0268$, and the spectral quotients $\beth_{1}=1$, $\beth_{2}=2.982$.
The invariant manifold, the spectral points and a representation of
the invariant vector bundle along the periodic orbit can be seen in
figure \ref{fig:shawpierre-spectrum}. 
\begin{figure}
\begin{centering}
\includegraphics[width=0.9\textwidth]{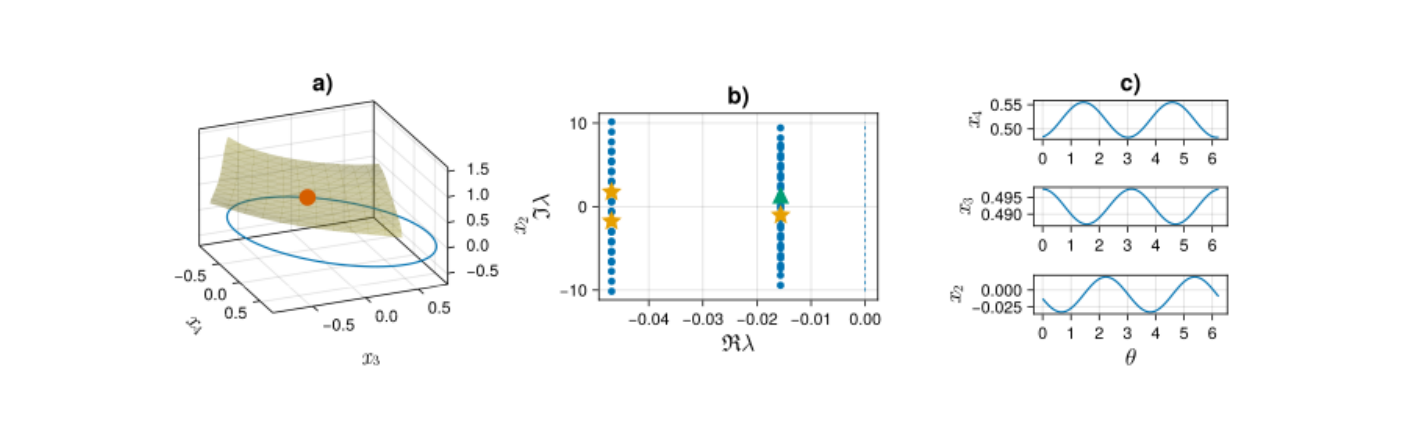}
\par\end{centering}
\caption{\label{fig:shawpierre-spectrum}a) The invariant torus and the invariant
manifold at one point along the torus. b) The spectrum of the linear
dynamics about the torus. The stars denote the representative eigenvalues
as selected by \eqref{eq:BUND-sel}, the triangle denotes the eigenvalue
used for model reduction. c) An incomplete representation of the invariant
vector bundle: three out of four coordinates of the first vector that
spans the two-dimensional invariant vector bundle.}
\end{figure}

For ROM identification, the data was generated using $600$ trajectories
each having $50$ points with time step $\Delta t=0.8$. The initial
conditions were sampled uniformly from the four dimensional unit ball.
The reduced order model is visualised in figure \ref{fig:shawpierre-ROM}.
The figures compare our direct calculation from the vector field and
the ROM identification (as per section \ref{subsec:ROMid}) from simulation
data. The instantaneous frequencies and damping ratios, as calculated
by the two methods, are close to each other for lower amplitudes and
diverge for higher amplitudes. The relative error of the asymptotic
expansion directly from the vector field varies with amplitude, while
the data-driven method has a uniform error over the amplitude. This
means that the direct calculation can have higher errors at higher
amplitudes than the data-driven method. There could be a few reasons
for disagreeing calculations at high amplitudes: non-uniqueness of
invariant manifolds discussed in the introduction, non-convergence
of series expansion (being outside the radius of convergence) or inability
to represent the foliation with functions satisfying the constraints
\eqref{eq:FOIL-constraints}. The most likely culprit is non-uniqueness,
although we have observed through carrying out varying high-order
direct calculations that the radius of convergence is only slightly
greater than the amplitude where the disagreement of the two calculations
occur. 
\begin{figure}
\begin{centering}
\includegraphics[width=0.9\textwidth]{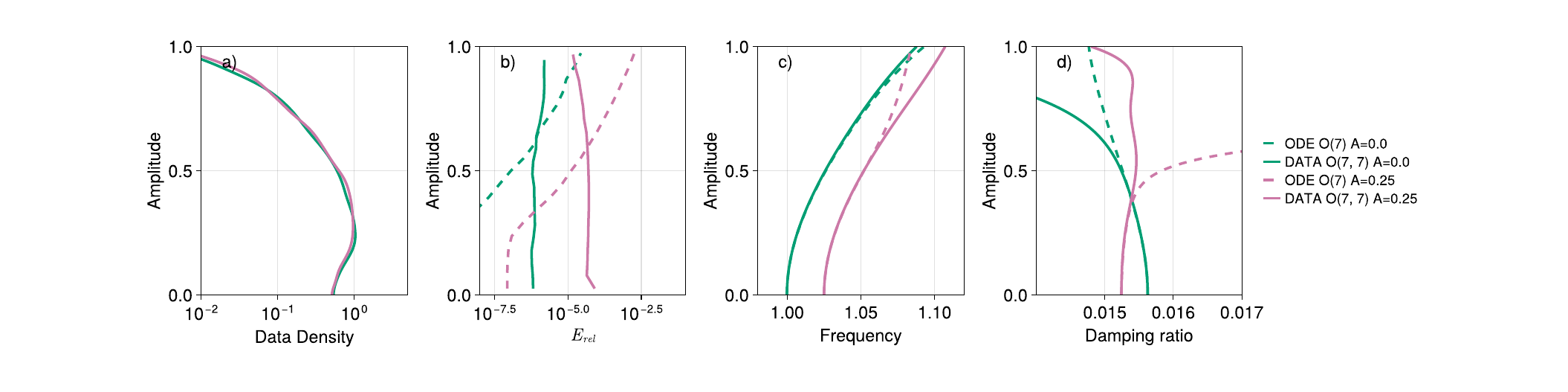}\\
\includegraphics[width=0.9\textwidth]{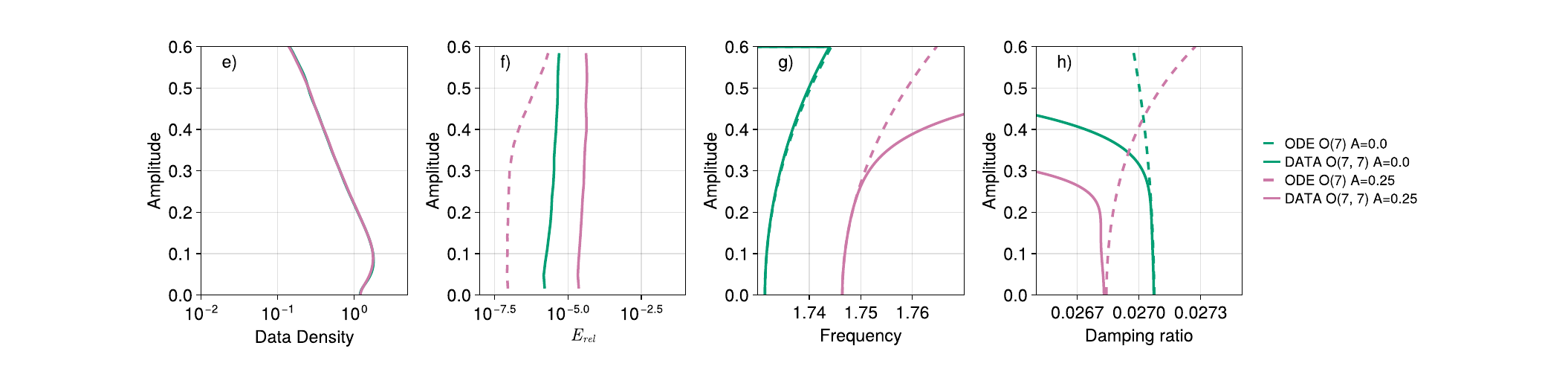}
\par\end{centering}
\caption{\label{fig:shawpierre-ROM}Reduced order models of \eqref{eq:ode-shawpierre}.
a,e) The distribution of data along the leaves of the foliation; b,f)
the relative error of the invariance equation over the invariant manifold
and over the data; c,g) instantaneous frequency of the vibration predicted
by the ROM; d,h) instantaneous damping ratio.}
\end{figure}

\subsection{Traffic model on a circular track}

We are using a car-following model with a simple optimal velocity
function as presented in \cite{OroszWilsonEtAl2009}. The model describes
a set of cars running on a circular track. The cars cannot overtake
each other and their velocity $v_{k}$ is dictated by the distance
from the car just in front of them $h_{k}$ (called headway). In this
particular model drivers reacts instantly to the changes in headway.
Each car has a maximum velocity $V_{k}$. The differential equations
describing the system are

\begin{align*}
\dot{v}_{k} & =\alpha\left(\frac{V_{k}\left(h_{k}-1\right)^{2}}{1+\left(h_{k}-1\right)^{2}}-v_{k}\right) &  & k=1\ldots n,\\
\dot{h}_{k} & =v_{k-1}-v_{k} &  & k=2\ldots n,\\
h_{1} & =L-\sum_{j=2}^{n}h_{j}.
\end{align*}
We assume that the length of the circular track is $L=2n$, the maximum
velocity of each vehicle is $V_{k}=1$, except for $V_{n}=1+A\cos\omega t$.
We set the forcing frequency $\omega=0.4374$ and use either $A=0$
or $A=0.2$ as the forcing amplitude. We also assume $n=5$ cars so
that semi-analytical calculations are still possible. For this given
set of parameters and without forcing, the steady state is such that
all headways are equal $h_{k}^{\star}=L/n$ and all velocities are
\[
v_{k}^{\star}=\frac{\left(L/n-1\right)^{2}}{1+\left(L/n-1\right)^{2}}.
\]
The value $\alpha=0.75$ is used so that the system is still stable,
but near the stability boundary. The eigenvalues of the Jacobian about
this equilibrium are
\begin{align*}
\lambda_{12} & =-0.0163\pm0.4971i,\\
\lambda_{34} & =-0.2276\pm0.7480i,\\
\lambda_{56} & =-0.5223\pm0.7480i,\\
\lambda_{78} & =-0.7337\pm0.4971i,\\
\lambda_{9} & =-0.75.
\end{align*}
The spectral quotient for the linear subspace is $\beth_{1-2}=1$
while the spectral quotient for the $\beth_{3-9}=45.993$. To discretise
the system we used $\ell=7$ Fourier harmonics and order 7 polynomials.
The forced system about the periodic orbit has representative eigenvalues
\begin{align*}
\lambda_{12} & =-0.0209\pm0.5040i,\\
\lambda_{34} & =-0.2281\pm0.7409i,\\
\lambda_{56} & =-0.5203\pm0.7409i,\\
\lambda_{78} & =-0.7306\pm0.4919i,\\
\lambda_{9} & =-0.75.
\end{align*}
The spectral quotients are $\beth_{1-2}=1$ and $\beth_{3-9}=33.7454$.
The invariant manifold about a point along the periodic orbit, the
spectrum and a representation of the invariant vector bundle can be
seen in figure \ref{fig:car-spectrum}.
\begin{figure}
\begin{centering}
\includegraphics[width=0.9\textwidth]{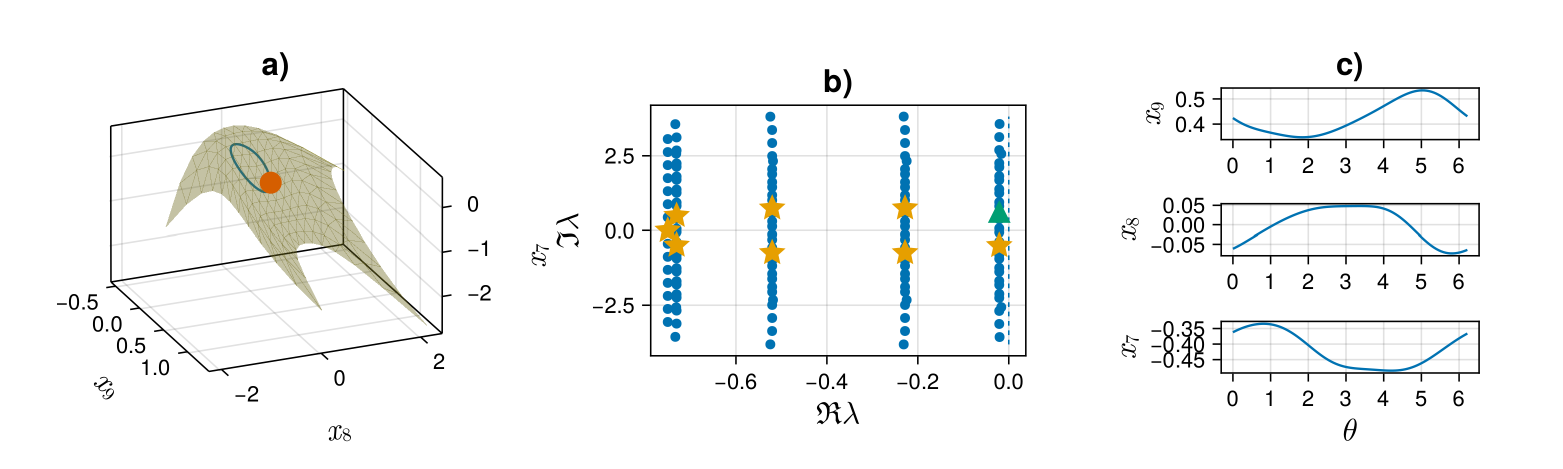}
\par\end{centering}
\caption{\label{fig:car-spectrum}a) The invariant torus and the invariant
manifold at one point along the torus. b) The spectrum of the linear
dynamics about the torus. The stars denote the representative eigenvalues
as selected by \eqref{eq:BUND-sel}, the triangle denotes the eigenvalue
use for model reduction. c) An incomplete representation of the invariant
vector bundle: three out of four coordinates of the first vector that
spans the two-dimensional invariant vector bundle.}
\end{figure}
For ROM identification, the data was generated using $600$ trajectories
each having $50$ points with time step $\Delta t=0.8$. The initial
conditions were sampled uniformly from the nine-dimensional ball of
radius $1.2$ about the origin. The ROM associated with the spectrum
points $\lambda_{12}$, both for the forced and unforced system can
be seen in figure \ref{fig:car-ROMs}. As before, the direct calculations
and the identified ROMs agree well for lower amplitudes, after which
they diverge. The divergence is most likely due to non-uniqueness
of invariant manifolds. The levels of error is higher than in the
previous example, which can also be attributed to the use of polynomials
with compressed tensor coefficients.
\begin{figure}
\begin{centering}
\includegraphics[width=0.99\textwidth]{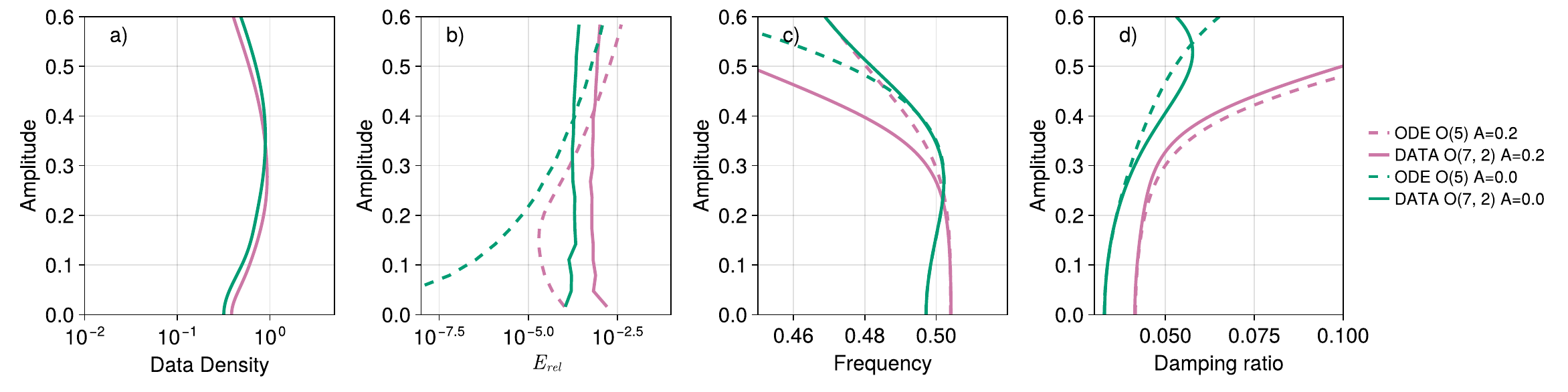}
\par\end{centering}
\caption{\label{fig:car-ROMs}Reduced order models of \eqref{eq:ode-shawpierre}.
a) The distribution of data along the leaves of the foliation; b)
the relative error of the invariance equation over the invariant manifold
and over the data; c) instantaneous frequency of the vibration predicted
by the ROM; d) instantaneous damping ratio.}
\end{figure}

\section{Discussion}

We have demonstrated how to identify reduced order models from data
using invariant foliations. We have recalled, that if genuine ROMs
are sought only invariant foliations can be used. We have discussed
how the uniqueness criteria of invariant manifolds and foliations
are local to the (quasi-) periodic orbit of the equilibrium and therefore
it is impossible to make use of when a ROM is identified from data.
This necessitates further work on finding a suitable and data-friendly
uniqueness criterion. Despite this issue, we were able to identify
reasonable ROMs at moderate distances from the equilibria and (quasi-)
periodic orbits.

In this work we emphasised the use of normal forms for ROM creation
as opposed to sparse regression, where the number of model components
are singled out using $L_{1}$ penalty terms within an optimisation
process \cite{BruntonPNAS2016}. Our approach therefore does not require
any trade-off between accuracy and model simplicity. The model is
always simple, within a suitable coordinate system.

In order to use autoencoders, the data has to lie on a low-dimensional
invariant manifold. Running simulations and ignoring the initial transition,
as in \cite{Cenedese2022NatComm} can be a strategy. This method however
has serious flaws. To be able to ignore the initial transient the
dynamics of interest must be significantly slower than the rest of
the dynamics. However when there is a very slow dynamics, the problem
non-uniqueness becomes significant given the large spectral quotient
of the invariant manifold. The data can also be skewed and have bias.
In fact, autoencoders identify where the data is in the phase space
as opposed where the sought after dynamics is.

In the companion paper \cite{SzalaiForcedTheory2024} we have also
touched on non-linearisability. In general, the dynamics about an
asymptotically stable (quasi-) periodic orbit or equilibrium is linearisable
within the basin of attraction. Therefore in \cite{SzalaiForcedTheory2024}
a linear ROM within a nonlinear coordinate system could capture nonlinear
phenomena remarkably well. However in a data-driven setting, the solution
of the invariance equation is not asymptotic but global and approximation
errors behave differently. In \cite{Szalai2023Fol} we have shown
that this approximation error causes the data-driven ROM to underestimate
the nonlinear dynamics. The present approach of invariant foliations
is suitable to capture fully nonlinear dynamics.

\paragraph{Software}

The computer code that produced the results in this paper can be found
at \href{https://github.com/rs1909/InvariantModels}{https://github.com/rs1909/InvariantModels}

\bibliographystyle{plain}
\bibliography{../../../Bibliography/AllRef}

\begin{thebibliography}{10}

\bibitem{AulbachFiber1998}
B.~Aulbach.
\newblock The fundamental existence theorem on invariant fiber bundles.
\newblock {\em Journal of Difference Equations and Applications},
  3(5-6):267--312, 1998.

\bibitem{BartelsStewartAlg432}
R.~H. Bartels and G.~W. Stewart.
\newblock Algorithm 432 [c2]: Solution of the matrix equation ax + xb = c [f4].
\newblock {\em Commun. ACM}, 15(9):820--826, 1972.

\bibitem{BatesFoliations2000}
P.~W. Bates, K.~Lu, and C.~Zeng.
\newblock Invariant foliations near normally hyperbolic invariant manifolds for
  semiflows.
\newblock {\em Trans. Amer. Math. Soc.}, 352(10):4641--4676, 2000.

\bibitem{biau2015lectures}
G.~Biau and L.~Devroye.
\newblock {\em Lectures on the Nearest Neighbor Method}.
\newblock Springer Series in the Data Sciences. Springer, 2015.

\bibitem{boyd_vandenberghe_2018}
S.~Boyd and L.~Vandenberghe.
\newblock {\em Introduction to Applied Linear Algebra: Vectors, Matrices, and
  Least Squares}.
\newblock Cambridge University Press, 2018.

\bibitem{BruntonPNAS2016}
S.L. Brunton, J.L. Proctor, J.N. Kutz, and W.~Bialek.
\newblock Discovering governing equations from data by sparse identification of
  nonlinear dynamical systems.
\newblock {\em Proceedings of the National Academy of Sciences of the United
  States of America}, 113(15):3932--3937, 2016.

\bibitem{CabreLlave2003}
X.~Cabr{\'e}, E.~Fontich, and R.~{de la Llave}.
\newblock The parameterization method for invariant manifolds {I}: {M}anifolds
  associated to non-resonant subspaces.
\newblock {\em Indiana Univ. Math. J.}, 52:283--328, 2003.

\bibitem{CAMPSVALLS20231}
G.~Camps-Valls, A.~Gerhardus, U.~Ninad, G.~Varando, G.~Martius,
  E.~Balaguer-Ballester, R.~Vinuesa, E.~Diaz, L.~Zanna, and J.~Runge.
\newblock Discovering causal relations and equations from data.
\newblock {\em Physics Reports}, 1044:1--68, 2023.

\bibitem{Cenedese2022NatComm}
M.~Cenedese, J.~Ax{\aa}s, B.~B{\"a}uerlein, K.~Avila, and G.~Haller.
\newblock Data-driven modeling and prediction of non-linearizable dynamics via
  spectral submanifolds.
\newblock {\em Nat Commun}, 13(872), 2022.

\bibitem{Champion2019Autoencoder}
K.~Champion, B.~Lusch, J.~{Nathan Kutz}, and S.~L. Brunton.
\newblock Data-driven discovery of coordinates and governing equations.
\newblock {\em Proceedings of the National Academy of Sciences of the United
  States of America}, 116(45):22445--22451, 2019.

\bibitem{chicone2008ordinary}
C.~Chicone.
\newblock {\em Ordinary Differential Equations with Applications}.
\newblock Texts in Applied Mathematics. Springer New York, 2008.

\bibitem{delaLlave1997}
R.~de~la Llave.
\newblock Invariant manifolds associated to nonresonant spectral subspaces.
\newblock {\em Journal of Statistical Physics}, 87(1):211--249, 1997.

\bibitem{Llave1995}
R.~de~la Llave and C.~E. Wayne.
\newblock On irwin's proof of the pseudostable manifold theorem.
\newblock {\em Mathematische Zeitschrift}, 219:301--321, 1995.

\bibitem{epstein2008}
J.~M. Epstein.
\newblock Why model?
\newblock {\em Journal of Artificial Societies and Social Simulation},
  11(4):12, 2008.

\bibitem{Haro2016}
{\`A}.~Haro, M.~Canadell, Al. Luque, J.~M. Mondelo, and J.-L. Figueras.
\newblock {\em The Parameterization Method for Invariant Manifolds: From
  Rigorous Results to Effective Computations}, volume 195 of {\em Applied
  Mathematical Sciences}.
\newblock Springer, 2016.

\bibitem{hirsch1970}
M.~W. Hirsch, C.~C. Pugh, and M.~Shub.
\newblock Invariant manifolds.
\newblock {\em Bull. Amer. Math. Soc.}, 76(5):1015--1019, 09 1970.

\bibitem{Kevrekidis2003}
I.~G. Kevrekidis, C.~W. Gear, J.~M. Hyman, P.~G. Kevrekidis, O.~Runborg, and
  C.~Theodoropoulos.
\newblock Equation-free, coarse-grained multiscale computation: enabling
  microscopic simulators to perform system-level analysis.
\newblock {\em Communications in Mathematical Sciences}, 1(4):715--762, 2003.

\bibitem{Kramer1991autoencoder}
M.~A. Kramer.
\newblock Nonlinear principal component analysis using autoassociative neural
  networks.
\newblock {\em AIChE Journal}, 37(2):233--243, 1991.

\bibitem{Lawson1974}
H.~Blaine Lawson, Jr.
\newblock Foliations.
\newblock {\em Bull. Amer. Math. Soc.}, 80:369--418, 1974.

\bibitem{GaussSouthwell2015}
J.~Nutini, M.~Schmidt, I.~H. Laradji, M.~Friedlander, and H.~Koepke.
\newblock Coordinate descent converges faster with the gauss-southwell rule
  than random selection.
\newblock In {\em Proceedings of the 32nd International Conference on
  International Conference on Machine Learning - Volume 37}, ICML'15, pages
  1632--1641, 2015.

\bibitem{OroszWilsonEtAl2009}
G.~Orosz, R.~E. Wilson, R.~Szalai, and G.~St{\'e}p{\'a}n.
\newblock Exciting traffic jams: Nonlinear phenomena behind traffic jam
  formation on highways.
\newblock {\em Phys. Rev. E}, 80:046205, Oct 2009.

\bibitem{ShubFoliation2012}
C.~Pugh, M.~Shub, and A.~Wilkinson.
\newblock H{\"o}lder foliations, revisited, 2012.

\bibitem{Roberts89}
A.~J. Roberts.
\newblock Appropriate initial conditions for asymptotic descriptions of the
  long-term evolution of dynamical systems.
\newblock {\em J. Austral. Math. Soc. Ser. B}, 31:48--75, 1989.

\bibitem{ShannonNyquist}
C.E. Shannon.
\newblock Communication in the presence of noise.
\newblock {\em Proceedings of the IRE}, 37(1):10--21, 1949.

\bibitem{ShawPierre}
S.~W. Shaw and C~Pierre.
\newblock Normal-modes of vibration for nonlinear continuous systems.
\newblock {\em {J. Sound Vibr.}}, {169}({3}):319--347, {1994}.

\bibitem{Szalai2020ISF}
R.~Szalai.
\newblock Invariant spectral foliations with applications to model order
  reduction and synthesis.
\newblock {\em Nonlinear Dynamics}, 101(4):2645--2669, 2020.

\bibitem{Szalai2023Fol}
R.~Szalai.
\newblock Data-driven reduced order models using invariant foliations,
  manifolds and autoencoders.
\newblock {\em Journal of Nonlinear Science}, 33(5):75, 2023.

\bibitem{SzalaiForcedTheory2024}
R.~Szalai.
\newblock Non-resonant invariant foliations of quasi-periodically forced
  systems, 2024.

\bibitem{trefethen}
L.~N. Trefethen.
\newblock {\em Spectral Methods in MATLAB}.
\newblock SIAM Philadelphia, 2000.

\bibitem{wiggins2003introduction}
S.~Wiggins.
\newblock {\em Introduction to Applied Nonlinear Dynamical Systems and Chaos}.
\newblock Texts in Applied Mathematics. Springer New York, 2003.

\end{thebibliography}

\end{document}